\documentclass[11pt]{article}
\usepackage[a4paper,left=2.45cm,right=2.45cm,top=2.25cm,bottom=2.35cm]{geometry}
\usepackage{amsmath,amssymb,amsthm,mathtools,mathrsfs,bm}
\usepackage{enumitem}
\usepackage[nocompress]{cite}
\usepackage{xcolor}
\usepackage[colorlinks=true,linkcolor=blue,citecolor=blue,urlcolor=blue,linktoc=all,bookmarksnumbered=true]{hyperref}
\usepackage{microtype}
\hypersetup{
  pdftitle={Enhanced Dissipation by Transport Noise for Stochastic p(x)-Laplace Equations},
  pdfauthor={Jiangwei Zhang, Juntao Wu, Shang Wu}
}
\allowdisplaybreaks[4]
\newtheorem{theorem}{Theorem}[section]
\newtheorem{proposition}[theorem]{Proposition}
\newtheorem{lemma}[theorem]{Lemma}
\newtheorem{corollary}[theorem]{Corollary}
\theoremstyle{definition}
\newtheorem{definition}[theorem]{Definition}
\newtheorem{assumption}[theorem]{Assumption}
\theoremstyle{remark}
\newtheorem{remark}[theorem]{Remark}
\numberwithin{equation}{section}

\newcommand{\dd}{\,\mathrm d}
\newcommand{\essinf}{\operatorname*{ess\,inf}}
\newcommand{\esssup}{\operatorname*{ess\,sup}}
\newcommand{\Div}{\operatorname{div}}

\newcommand{\norm}[1]{\left\lVert #1\right\rVert}
\newcommand{\abs}[1]{\left|#1\right|}
\newcommand{\pair}[2]{\left\langle #1,#2\right\rangle}
\newcommand{\Dpx}{\Delta_{p(\cdot)}}

\title{\bfseries\Large Enhanced Dissipation by Transport Noise for Stochastic $p(x)$-Laplace Equations\footnote{\small\textit{E-mail addresses:} zjwmath@163.com (J. Zhang), wu\_jun\_tao@163.com (J. Wu), ws970303@163.com (S. Wu)}}

\author{{Jiangwei Zhang$^\text{a}\footnote{Corresponding author.}$, \quad Juntao Wu$^\text{b}$, \quad Shang Wu$^\text{c}$
}\\
{\small\textsl{$^\text{a}$ Institute of Applied Physics and Computational Mathematics, }}\\
{\small\textsl{Beijing 100088, P.R. China}}\\
{\small\textsl{$^\text{b}$ School of Mathematics and Statistics, Wuhan University, }}\\
{\small\textsl{Wuhan, Hubei 430072, P.R. China}}\\
{\small\textsl{$^\text{c}$ Intelligent Game and Decision Lab, Academy of Military Sciences, }}\\
{\small\textsl{Beijing 100071, P.R. China}}
}
\date{}

\begin{document}
\maketitle

\begin{abstract}
We are concerned with dissipation enhancement by divergence-free transport noise for stochastic \(p(x)\)-Laplace equations on \(\mathbb T^d\), \(d\ge2\). Spatially variable growth removes homogeneous scaling and produces two flux growth regimes. For periodic log-H\"older continuous exponents satisfying \(2<p^-\le p^+<\infty\), we first establish global variational well-posedness and the exact energy identity in the natural dual spaces. Under the additional condition \(p^+<2d/(d-2)\) when \(d\ge3\), we then prove averaged algebraic decay with any prescribed coefficient, uniformly on bounded zero-mean \(L^2\) balls, by estimating the two flux powers separately. A deterministic lower bound yields strict averaged comparison with the unperturbed flow at prescribed positive times on shells with bounded gradient modular. We also prove optimality of the modular coercivity exponent \(p^+\) near zero. Finally, combining these estimates with intervalwise noise acceleration from the constant-exponent theory gives exponential decay in mean square and almost surely, with moments of the random prefactors bounded uniformly on the initial-data ball. For the exponential estimates, the noise depends on the prescribed rate and moment order; its intensity is fixed in time and its frequency range increases.


\end{abstract}

\section{Introduction}

\subsection{Background and formulation}

Let $d\ge2$ and let $\mathbb T^d=\mathbb R^d/\mathbb Z^d$, with $|\mathbb T^d|=1$. We consider
\begin{equation}\label{eq:intro-main}
\begin{cases}
\dd u=\Dpx u\,\dd t+\nabla u\circ\dd W(t,x),&t>0,\ x\in\mathbb T^d,\\
u(0,x)=u_0(x),&x\in\mathbb T^d,
\end{cases}
\end{equation}
where
\[
\Dpx u=\Div\bigl(|\nabla u|^{p(x)-2}\nabla u\bigr).
\]
The field $W$ is white in time, colored in space and divergence-free. Its Fourier representation is given in Section~\ref{sec:noise}. Equation~\eqref{eq:intro-main} is understood through the It\^o variational equation~\eqref{eq:ito-main}; the circle denotes the formal Stratonovich transport term.

Variable-exponent diffusion arises in constitutive laws with spatially dependent growth. Its variational theory is developed in~\cite{Zhikov1987,RuzickaBook,DieningBook}. For such operators, the dissipation is measured by
\begin{equation}\label{eq:intro-modular}
\rho_{p(\cdot)}(f):=\int_{\mathbb T^d}|f(x)|^{p(x)}\dd x.
\end{equation}
Stochastic equations with nonstandard growth have been studied in~\cite{BauzetValletWittboldZimmermann2013,ValletZimmermann2016,ZhangWuHuang2025,ZhangQiuHuang2026}. Here we study the effect of transport noise on the decay of solutions.

We impose the following condition on the exponent.

\begin{assumption}\label{ass:p}
The exponent $p:\mathbb T^d\to(1,\infty)$ admits a periodic representative, still denoted by $p$, for which there exists $C_{\log}>0$ such that
\begin{equation*}
\abs{p(x)-p(y)}
\le \frac{C_{\log}}{\log\!\left(e+d_{\mathbb T}(x,y)^{-1}\right)}
\qquad (x\ne y),
\end{equation*}
where $d_{\mathbb T}$ is the distance on the torus.  In particular, this representative is continuous.  Moreover,
\begin{equation*}
2<p^-:=\essinf_{x\in\mathbb T^d}p(x)
\le p(x)\le
p^+:=\esssup_{x\in\mathbb T^d}p(x)<p_{\mathrm c},
\end{equation*}
where
\[
p_{\mathrm c}:=
\begin{cases}
\dfrac{2d}{d-2},&d\ge3,\\[1ex]
\infty,&d=2.
\end{cases}
\]
\end{assumption}

The log-H\"older condition gives smooth density and the Poincar\'e inequality in the variable-exponent spaces. Both diffusion and transport preserve the spatial mean. We therefore use the zero-mean spaces $H$ and $V$ defined in Section~\ref{sec:prelim}. For the Fourier noise constructed below, Theorem~\ref{thm:wellposed} gives
\[
\|u(t)\|_{L^2}^2+2\int_0^t\rho_{p(\cdot)}(\nabla u(s))\dd s
=\|u_0\|_{L^2}^2.
\]
Thus the transport noise is $L^2$ conservative. Its It\^o correction is $\kappa\Delta$, and the associated heat semigroup will be used to estimate the nonlinear and stochastic convolutions.

Set
\begin{equation}\label{eq:eta-def-intro}
\eta:=\frac12+\frac d2\left(\frac12-\frac1{p^+}\right).
\end{equation}
The condition $p^+<p_{\mathrm c}$ gives $\eta<1$, so the kernel $(t-s)^{-\eta}$ is integrable. This upper restriction is used in the convolution estimates. The variational construction in Subsection~\ref{sec:wellposed-main} only requires the log-H\"older condition and $2<p^-\le p^+<\infty$.

\subsection{Related results and main difficulties}

Enhanced dissipation by incompressible transport has been studied in~\cite{ConstantinKiselevRyzhikZlatos2008,FengIyer2019,FengHuXu}. Results for stochastic transport include scaling limits, quantitative estimates and pathwise decay; see~\cite{Galeati2020,FlandoliGaleatiLuo2021,FlandoliGaleatiLuo2024,BedrossianBlumenthalPunshonSmith2021,GessYaroslavtsev2025,FlandoliLuo2026}.

For constant $p>2$, Dong, Luo and Tang~\cite{DongLuoTang2023} proved averaged and pathwise dissipation enhancement uniformly on bounded initial-data balls. Their argument combines heat smoothing, a large- and small-dissipation decomposition, and the distribution of the noise over many Fourier modes. Their Proposition~5.1 already implies exponential mean-square decay, and the discussion following their Theorem~1.2 explains how a larger geometric factor increases the admissible moment range. We use this acceleration mechanism after establishing estimates compatible with spatially variable growth.

The first issue is the duality needed for the energy identity. The gradient and flux are controlled in $L^{\widetilde p}(Q_T)$ and $L^{\widetilde p'}(Q_T)$, where $\widetilde p(\omega,t,x)=p(x)$. Replacing $p(x)$ by its extrema gives only $\nabla u\in L^{p^-}(Q_T)$ and $F\in L^{(p^+)'}(Q_T)$, which are not a conjugate pair when $p^-<p^+$. We retain the variable-exponent pairing and use a compatible Galerkin correction to preserve transport cancellation. Appendix~\ref{sec:appendix-framework} adapts the time discretization of~\cite[Theorem~A.2]{GwiazdaWoznickiWroblewskaZimmermann2025} to prove the energy identity before Minty identification. The conditions needed for this adaptation are specified there.

The second issue is the loss of homogeneous scaling. For $\hat{D}=\rho_{p(\cdot)}(\nabla u)$, Lemma~\ref{lem:flux-norm} gives
\[
\|J_{p(\cdot)}(u)\|_{L^{p'(\cdot)}}
\le \hat{D}^{(p^--1)/p^-}+\hat{D}^{(p^+-1)/p^+}.
\]
Lemma~\ref{lem:v2} estimates these powers separately by heat smoothing and the pathwise energy identity. In the stochastic convolution, the small-dissipation bound produces both $X$ and $X^{p^+/p^-}$, where $X=\mathbb E\|u(s)\|_{L^2}^2$. On $X\le R^2$, we use
\[
X^{p^+/p^-}\le R^{2(p^+/p^- -1)}X.
\]
Together with integration by parts and Bessel's inequality, this preserves the factor $\|\theta\|_{\ell^\infty}^{2/\beta}$ in Lemma~\ref{lem:v3} and permits a noise choice uniform on the initial-data ball. These estimates extend the averaged enhancement argument of~\cite{DongLuoTang2023}; their intervalwise acceleration then yields the prescribed exponential rate and moment order.

For strict comparison, we prove that
\[
t\longmapsto
\left(\int_{\mathbb T^d}\frac{|\nabla u(t,x)|^{p(x)}}{p(x)}\dd x\right)
\|u(t)\|_{L^2}^{-p^-}
\]
is nonincreasing along the nonzero deterministic flow. The resulting lower bound gives strict averaged comparison at any prescribed positive time on shells bounded away from zero with bounded gradient modular. We also prove that $p^+$ is the smallest exponent in a uniform modular lower bound near zero. This coercivity statement does not determine the asymptotic decay exponent of an individual solution.
\subsection{Main results}
We first state averaged enhancement and strict comparison for time-independent profiles, then exponential decay for stepwise profiles and modular sharpness. Corollary~\ref{cor:path-main} gives a further pathwise consequence.

\begin{theorem}[Averaged enhanced dissipation]\label{thm:average-main}
Suppose that Assumption~\ref{ass:p} holds. Fix $R>0$ and set $m=p^+$. For every $\mu>0$, there exist $\kappa>0$ and a normalized finitely supported annular profile $\theta$ satisfying~\eqref{eq:theta-assumptions} such that, for every $u_0\in H$ with $\|u_0\|_{L^2}\le R$, the variational solution of~\eqref{eq:ito-main} satisfies
\begin{equation}\label{eq:average-main}
\norm{u(t)}_{L^2(\Omega;L^2)}
\le
\frac{\norm{u_0}_{L^2}}
{\left(1+(m-2)\mu t\norm{u_0}_{L^2}^{m-2}\right)^{1/(m-2)}}
\qquad\text{for every }t\ge1.
\end{equation}
\end{theorem}

\begin{remark}\label{rem:average-scope}
The noise is chosen uniformly on the initial-data ball. The coefficient $\mu$ may exceed $c_R^{\mathrm{opt}}$ from Remark~\ref{rem:optimal-coercivity-constant}. The restriction $t\ge1$ arises only in passing from grid times to arbitrary times; Theorem~\ref{thm:strict-enhancement} gives a direct comparison at any prescribed $T>0$.
\end{remark}

The deterministic lower bound in Theorem~\ref{thm:lower-bound} yields a strict comparison with the unperturbed evolution.

\begin{theorem}[Strict enhancement at a prescribed time]\label{thm:strict-enhancement}
Suppose that Assumption~\ref{ass:p} holds and let $0<r_*\le R$ and $M>0$.  Put
\[
\mathcal S(r_*,R,M):=\Bigl\{u_0\in W^{1,p(\cdot)}(\mathbb T^d):\ \textstyle\int_{\mathbb T^d}u_0\dd x=0,\
r_*\le\norm{u_0}_{L^2}\le R,\ \rho_{p(\cdot)}(\nabla u_0)\le M\Bigr\}.
\]
For every $T>0$ and $\lambda\in(0,1)$, there exist $\kappa>0$ and a finitely supported annular profile $\theta$ satisfying~\eqref{eq:theta-assumptions}, depending only on $T,\lambda,r_*,R,M$ and the structural data, such that, for every $u_0\in\mathcal S(r_*,R,M)$,
\begin{equation*}
\norm{u(T)}_{L^2(\Omega;L^2)}\le\lambda\norm{u^{\mathrm{det}}(T)}_{L^2}.
\end{equation*}
Here $u$ solves~\eqref{eq:ito-main} and $u^{\mathrm{det}}$ solves~\eqref{eq:unperturbed} with the same initial datum.
\end{theorem}

\begin{theorem}[Exponential decay]\label{thm:exponential-main}
Suppose that Assumption~\ref{ass:p} holds. Given $R,c,q>0$, there exist an intensity $\kappa>0$, fixed in time, and deterministic stepwise profiles, finitely supported on each interval and satisfying~\eqref{eq:theta-assumptions}, such that, for every $u_0\in H$ with $\|u_0\|_{L^2}\le R$, the solution of~\eqref{eq:time-dependent-spde} satisfies
\begin{equation}\label{eq:exponential-mean-main}
\mathbb E\|u(t)\|_{L^2}^2\le \widehat C e^{-2ct},\qquad t\ge0,
\end{equation}
where $\widehat C$ is independent of $u_0$, and
\begin{equation}\label{eq:exponential-path-main}
\|u(t)\|_{L^2}\le K(u_0,\omega)e^{-ct},\qquad t\ge0,\quad\text{almost surely},
\end{equation}
with
\begin{equation}\label{eq:exponential-moment-main}
\sup_{\substack{u_0\in H\\\|u_0\|_{L^2}\le R}}
\mathbb E[K(u_0,\cdot)^q]<\infty.
\end{equation}
The intensity and profiles depend only on $R,c,q$ and the structural data and are chosen before the initial datum. 
The profiles can be chosen to be annular, with the frequency
growth described in Remark~\ref{rem:frequency-growth}.
\end{theorem}

\begin{remark}\label{rem:pathwise-scope}
The profiles are fixed on each $[n,n+1)$ and move to larger frequency annuli. Their intensity is fixed in time, while their spatial derivatives grow as quantified in Remark~\ref{rem:frequency-growth}. The noise is chosen for the prescribed rate $c$ and moment order $q$. Estimate~\eqref{eq:exponential-path-main} holds for each fixed datum, and~\eqref{eq:exponential-moment-main} is uniform on the initial-data ball; a common random prefactor for all data is not asserted. The normalized prefactor in Corollary~\ref{cor:path-main} has uniform moments on shells bounded away from zero.
\end{remark}
\begin{proposition}[Sharpness of the modular coercivity exponent]\label{prop:pplus-optimality}
Suppose that Assumption~\ref{ass:p} holds. For every $0<q<p^+$ and $c>0$, there are nonzero zero-mean $u\in W^{1,p(\cdot)}(\mathbb T^d)$ with arbitrarily small $\|u\|_{L^2}$ such that
\[
\rho_{p(\cdot)}(\nabla u)<c\|u\|_{L^2}^q.
\]
Together with Lemma~\ref{lem:modular-lower}, this shows that $p^+$ is the smallest exponent admitting a uniform modular lower bound near zero.
\end{proposition}

\begin{remark}\label{rem:pplus-scope}
For the unperturbed flow, Corollary~\ref{cor:baseline-decay} and Theorem~\ref{thm:lower-bound} give algebraic upper and lower bounds with exponents $1/(p^+-2)$ and $1/(p^--2)$. They agree for constant $p$. For nonconstant $p$, these bounds and the modular sharpness result do not determine the asymptotic exponent of an individual solution.
\end{remark}

Section~\ref{sec:prelim} gives the variable-exponent estimates and deterministic decay bounds. Section~\ref{sec:noise} proves well-posedness and the mild formulation. Sections~\ref{sec:average}--\ref{sec:pathwise} establish the enhancement results, using the energy formula proved in Appendix~\ref{sec:appendix-framework}.

\section{Preliminaries}\label{sec:prelim}

\subsection{Lebesgue and Sobolev spaces with variable exponent}

Let $p$ satisfy Assumption~\ref{ass:p} and set $p'(x)=p(x)/(p(x)-1)$. With the modular in~\eqref{eq:intro-modular}, define the Luxemburg norm by
\begin{equation*}
\norm{f}_{L^{p(\cdot)}}
:=\inf\left\{\lambda>0:\rho_{p(\cdot)}(f/\lambda)\le1\right\}.
\end{equation*}

The space $L^{p(\cdot)}(\mathbb T^d)$ is reflexive and separable. The log-H\"older condition gives smooth density in $W^{1,p(\cdot)}(\mathbb T^d)$; see~\cite{DieningBook,RuzickaBook}.

We use the following spaces of functions with zero mean
\begin{align*}
H&:=\left\{u\in L^2(\mathbb T^d):\int_{\mathbb T^d}u\dd x=0\right\},\\
V&:=\left\{u\in W^{1,p(\cdot)}(\mathbb T^d):\int_{\mathbb T^d}u\dd x=0\right\},
\end{align*}
with norm $\norm{u}_V:=\norm{\nabla u}_{L^{p(\cdot)}}$. The variable exponent Poincar\'e inequality gives a constant $\hat{C}_p>0$ such that
\begin{equation}\label{eq:px-poincare}
\norm{u}_{L^{p(\cdot)}}\le \hat{C}_p\norm{\nabla u}_{L^{p(\cdot)}}
\qquad\text{for every }u\in V.
\end{equation}
Since $p^->2$ and $|\mathbb T^d|=1$, there is a constant $\hat{C}_e>0$ such that
\begin{equation}\label{eq:embedding-L2}
\norm{f}_{L^2}\le \hat{C}_e\norm{f}_{L^{p(\cdot)}}.
\end{equation}
In particular, the embedding $V\hookrightarrow H$ is continuous and compact.

\begin{lemma}\label{lem:modular-norm}
For $f\in L^{p(\cdot)}$ and $g\in L^{p'(\cdot)}$, writing $z=\|f\|_{L^{p(\cdot)}}$, one has
\begin{gather}
\min\{z^{p^-},z^{p^+}\}
\le\rho_{p(\cdot)}(f)\le\max\{z^{p^-},z^{p^+}\},\label{eq:modular-norm-1}\\
z\le\rho_{p(\cdot)}(f)^{1/p^-}+\rho_{p(\cdot)}(f)^{1/p^+},\label{eq:modular-norm-2}\\
\int_{\mathbb T^d}|fg|\dd x
\le2\|f\|_{L^{p(\cdot)}}\|g\|_{L^{p'(\cdot)}}.\label{eq:px-holder}
\end{gather}
\end{lemma}
\begin{proof}
For $z>0$, $\rho_{p(\cdot)}(f/z)=1$. Bounding $z^{p(x)}$ gives~\eqref{eq:modular-norm-1}, and its two regimes give~\eqref{eq:modular-norm-2}. Young's inequality applied to the normalized functions yields~\eqref{eq:px-holder}; see~\cite{DieningBook}.
\end{proof}

\subsection{The nonlinear flux and the \texorpdfstring{$p(x)$}{p(x)}-Laplace operator}

For $u\in V$, we define the nonlinear flux as follows:
\[
J_{p(\cdot)}(u):=\abs{\nabla u}^{p(x)-2}\nabla u,
\qquad
\Dpx u:=\Div J_{p(\cdot)}(u)\in V^*.
\]
Then $J_{p(\cdot)}(u)\in L^{p'(\cdot)}(\mathbb T^d;\mathbb R^d)$ and
\begin{equation}\label{eq:flux-modular}
\rho_{p'(\cdot)}\bigl(J_{p(\cdot)}(u)\bigr)
=\int_{\mathbb T^d}\abs{\nabla u}^{(p(x)-1)p'(x)}\dd x
=\rho_{p(\cdot)}(\nabla u).
\end{equation}

\begin{lemma}\label{lem:flux-norm}
Set
\begin{equation*}
a_-:=\frac{p^--1}{p^-},
\qquad
a_+:=\frac{p^+-1}{p^+}.
\end{equation*}
Then, for every $u\in V$,
\begin{equation*}
\norm{J_{p(\cdot)}(u)}_{L^{p'(\cdot)}}
\le
\rho_{p(\cdot)}(\nabla u)^{a_-}
+
\rho_{p(\cdot)}(\nabla u)^{a_+}.
\end{equation*}
\end{lemma}

\begin{proof}
Since $(p')^-=p^+/(p^+-1)$ and $(p')^+=p^-/(p^--1)$, applying~\eqref{eq:modular-norm-2} with exponent $p'(\cdot)$ and using~\eqref{eq:flux-modular} gives the result.
\end{proof}

\begin{proposition}\label{prop:px-properties}
For all $u,v\in V$,
\begin{equation}\label{eq:px-monotone}
\pair{\Dpx u-\Dpx v}{u-v}_{V^*,V}
=-\int_{\mathbb T^d}
\bigl(J_{p(\cdot)}(u)-J_{p(\cdot)}(v)\bigr)
\cdot\nabla(u-v)\dd x\le0,
\end{equation}
\begin{equation}\label{eq:px-coercive}
\pair{\Dpx u}{u}_{V^*,V}
=-\rho_{p(\cdot)}(\nabla u),
\end{equation}
and $\lambda\mapsto\pair{\Dpx(u+\lambda v)}{w}_{V^*,V}$ is continuous on $\mathbb R$ for $u,v,w\in V$. Moreover,

\begin{equation}\label{eq:px-dual-bound}
\norm{\Dpx u}_{V^*}
\le2\norm{J_{p(\cdot)}(u)}_{L^{p'(\cdot)}}
\le2\left(
\rho_{p(\cdot)}(\nabla u)^{a_-}
+\rho_{p(\cdot)}(\nabla u)^{a_+}
\right).
\end{equation}
\end{proposition}

\begin{proof}
Pointwise monotonicity and integration by parts give~\eqref{eq:px-monotone}--\eqref{eq:px-coercive}. For $|\lambda|\le L$, Young's inequality yields
\[
|\nabla(u+\lambda v)|^{p(x)-1}|\nabla w|
\le C_L\bigl(|\nabla u|^{p(x)}+|\nabla v|^{p(x)}+|\nabla w|^{p(x)}\bigr).
\]
Dominated convergence gives hemicontinuity. Finally,~\eqref{eq:px-holder} and Lemma~\ref{lem:flux-norm} give~\eqref{eq:px-dual-bound} by taking the supremum over $\|v\|_V\le1$.
\end{proof}

\subsection{A uniform modular lower bound}

\begin{lemma}\label{lem:modular-lower}
Let $R>0$ and $m=p^+$.  Every $u\in V$ satisfies the following estimate
\begin{equation}\label{eq:modular-two-regime}
\rho_{p(\cdot)}(\nabla u)
\ge
\min\left\{(\hat{C}_e\hat{C}_p)^{-m}\norm{u}_{L^2}^{m},1\right\}.
\end{equation}
Consequently, there exists
$
c_R:=\min\left\{(\hat{C}_e\hat{C}_p)^{-m},(1+R)^{-m}\right\}>0
$
such that every $u\in V$ with $\norm{u}_{L^2}\le R$ satisfies
\begin{equation}\label{eq:modular-lower}
\rho_{p(\cdot)}(\nabla u)
\ge c_R\norm{u}_{L^2}^{m}.
\end{equation}
\end{lemma}

\begin{proof}
Set $z=\|\nabla u\|_{L^{p(\cdot)}}$. By~\eqref{eq:px-poincare}--\eqref{eq:embedding-L2}, $\|u\|_{L^2}\le\hat C_e\hat C_pz$, and hence
\[
\rho_{p(\cdot)}(\nabla u)\ge
\begin{cases}
z^{p^+}\ge(\hat C_e\hat C_p)^{-m}\|u\|_{L^2}^m,&z\le1,\\
z^{p^-}\ge1,&z>1.
\end{cases}
\]
This gives~\eqref{eq:modular-two-regime}. If $\|u\|_{L^2}\le R$, then $1\ge(1+R)^{-m}\|u\|_{L^2}^m$, which proves~\eqref{eq:modular-lower}.
\end{proof}

\begin{remark}\label{rem:optimal-coercivity-constant}
The intrinsic coefficient
\begin{equation*}
c_R^{\mathrm{opt}}:=\inf_{\substack{u\in V\\0<\|u\|_{L^2}\le R}}
\frac{\rho_{p(\cdot)}(\nabla u)}{\|u\|_{L^2}^{p^+}}
\end{equation*}
satisfies $0<c_R\le c_R^{\mathrm{opt}}<\infty$ by Lemma~\ref{lem:modular-lower} and evaluation at a nonzero smooth function in the ball. It is the largest admissible coefficient in~\eqref{eq:modular-lower} and may replace $c_R$ in~\eqref{eq:baseline-decay}.

\end{remark}

\begin{proof}[Proof of Proposition~\ref{prop:pplus-optimality}]
By continuity of \(p\), we have \(\max p=p^+\). Given \(q<p^+\), choose \(\delta>0\) and a nonempty open set \(U\) such that \(p\ge q+\delta\) on \(U\). Let \(\psi\in C_c^\infty(U)\) be nonconstant and define \(\varphi=\psi-\int\psi\,\mathrm{d}x\). Then, for \(0<\varepsilon\le1\), the nonzero mean-zero function \(u_\varepsilon=\varepsilon\varphi\) satisfies
\[
\rho_{p(\cdot)}(\nabla u_\varepsilon)
=\int_U\varepsilon^{p(x)}|\nabla\varphi|^{p(x)}\dd x
\le\varepsilon^{q+\delta}\rho_{p(\cdot)}(\nabla\varphi).
\]
Therefore
\[
\frac{\rho_{p(\cdot)}(\nabla u_\varepsilon)}{\|u_\varepsilon\|_{L^2}^{q}}
\le\frac{\rho_{p(\cdot)}(\nabla\varphi)}{\|\varphi\|_{L^2}^{q}}\varepsilon^\delta
\longrightarrow0,\qquad \|u_\varepsilon\|_{L^2}\longrightarrow0.
\]
This completes the proof.
\end{proof}

\begin{corollary}\label{cor:baseline-decay}
Let $R>0$ and $m=p^+$. Suppose a solution satisfies the energy identity
\[
\norm{u(t)}_{L^2}^2+2\int_0^t\rho_{p(\cdot)}(\nabla u(s))\dd s
=\norm{u_0}_{L^2}^2
\]
and $\norm{u_0}_{L^2}\le R$.  Then
\begin{equation}\label{eq:baseline-decay}
\norm{u(t)}_{L^2}
\le
\frac{\norm{u_0}_{L^2}}
{\left(1+c_R(m-2)t\norm{u_0}_{L^2}^{m-2}\right)^{1/(m-2)}}.
\end{equation}
\end{corollary}

\begin{proof}
For $y=\|u\|_{L^2}^2$, the energy identity and Lemma~\ref{lem:modular-lower} give
\[
y'\le-2c_Ry^{m/2},\qquad
\frac{\dd}{\dd t}y^{-(m-2)/2}\ge c_R(m-2)\quad\text{where }y>0.
\]
Integrating and taking square roots gives~\eqref{eq:baseline-decay}. If \(y\) reaches zero, then norm monotonicity implies that \(u=0\) for all later times.
\end{proof}

\subsection{A lower bound for the unperturbed equation}\label{subsec:lower-bound}

We consider the following unperturbed equation on the zero-mean subspace:
\begin{equation}\label{eq:unperturbed}
\partial_t u=\Delta_{p(\cdot)}u,\qquad u(0)=u_0,
\end{equation}
and we write
\begin{equation*}
\mathcal J(u):=\int_{\mathbb T^d}\frac{\abs{\nabla u(x)}^{p(x)}}{p(x)}\dd x
\quad(u\in V),
\qquad \mathcal J:=+\infty \text{ on }H\setminus V .
\end{equation*}
Thanks to $p^-\le p(x)\le p^+$, we have
\begin{equation}\label{eq:J-rho-comparison}
\frac1{p^+}\,\rho_{p(\cdot)}(\nabla u)\;\le\;\mathcal J(u)\;\le\;\frac1{p^-}\,\rho_{p(\cdot)}(\nabla u).
\end{equation}
The functional $\mathcal J$ is convex, proper and lower semicontinuous on $H$, and $\partial\mathcal J(u)=-\Delta_{p(\cdot)}u$ whenever the right-hand side belongs to $H$.  By the theory of gradient flows of convex functionals in Hilbert space~\cite{Brezis1973}, for every zero-mean $u_0\in H$, equation~\eqref{eq:unperturbed} admits a unique strong solution $u\in C([0,\infty);H)$ which is locally Lipschitz from $(0,\infty)$ into $H$, satisfies $u(t)\in D(\partial\mathcal J)$ for every $t>0$, and obeys
\begin{equation}\label{eq:gradient-flow-identities}
\partial_tu(t)=\Delta_{p(\cdot)}u(t),
\qquad
\frac{\dd}{\dd t}\mathcal J(u(t))=-\norm{\partial_tu(t)}_{L^2}^2
\end{equation}
for almost every $t>0$.  If in addition $u_0\in V$, then $\mathcal J(u(t))\to\mathcal J(u_0)$ as $t\downarrow0$.

\begin{lemma}\label{lem:monotone-scaling}
Let $u_0\in V$ have zero mean and $u_0\ne0$, and let $u$ solve~\eqref{eq:unperturbed}.  Set
\[
y(t):=\norm{u(t)}_{L^2}^2,\qquad
\mathcal J(t):=\mathcal J(u(t)),\qquad
\rho(t):=\rho_{p(\cdot)}(\nabla u(t)).
\]
Let $T_*:=\sup\{\,\tau>0:\ y>0 \text{ on }[0,\tau)\,\}$.  Then $t\mapsto\mathcal J(t)\,y(t)^{-p^-/2}$ is nonincreasing on $[0,T_*)$.
\end{lemma}

\begin{proof}
By continuity, \(T_*>0\). On every compact subinterval of \((0,T_*)\), both \(y\) and \(\mathcal J\) are absolutely continuous. Testing \eqref{eq:unperturbed} against \(u\) and using \eqref{eq:gradient-flow-identities} yields
\begin{equation*}
y'=-2\rho,\qquad
\rho^2=|\langle\partial_tu,u\rangle|^2
\le\|\partial_tu\|_{L^2}^2y=-\mathcal J'y.
\end{equation*}
By~\eqref{eq:J-rho-comparison}, $p^-\mathcal J\le\rho$. Thus, almost everywhere on $(0,T_*)$,
\[
\frac{\dd}{\dd t}\bigl(\mathcal J y^{-p^-/2}\bigr)
=y^{-p^-/2-1}\bigl(y\mathcal J'+p^-\mathcal J\rho\bigr)
\le y^{-p^-/2-1}\rho(p^-\mathcal J-\rho)\le0.
\]
Integrating the above inequality from \(s\) to \(t<T_*\) and letting \(s\downarrow0\), while using
\(y(s)\to y(0)>0\) and \(\mathcal J(u(s))\to\mathcal J(u_0)\), we obtain
\[
\mathcal J(t)y(t)^{-p^-/2}\le \mathcal J(0)y(0)^{-p^-/2}.
\]
This implies the desired conclusion.
\end{proof}

\begin{theorem}\label{thm:lower-bound}
Let $u_0\in V$ have zero mean and $u_0\ne0$, and let $u$ be the solution of~\eqref{eq:unperturbed}.  Set
\begin{equation}\label{eq:def-Lambda0}
\Lambda_0(u_0):=\frac{p^+\,\mathcal J(u_0)}{\norm{u_0}_{L^2}^2}
\;\le\;\frac{p^+}{p^-}\cdot\frac{\rho_{p(\cdot)}(\nabla u_0)}{\norm{u_0}_{L^2}^2}\;<\;\infty.
\end{equation}
Then for every $t\ge0$, $u(t)\ne0$ and
\begin{equation}\label{eq:lower-bound}
\norm{u(t)}_{L^2}
\;\ge\;
\frac{\norm{u_0}_{L^2}}
{\bigl(1+(p^--2)\,\Lambda_0(u_0)\,t\bigr)^{1/(p^--2)}}.
\end{equation}
\end{theorem}

\begin{proof}
Use the notation of Lemma~\ref{lem:monotone-scaling} and set $y_0=\|u_0\|_{L^2}^2$. On $(0,T_*)$, that lemma and~\eqref{eq:J-rho-comparison} give
\[
-y'=2\rho\le2p^+\mathcal J(t)
\le2p^+\mathcal J(u_0)\left(\frac{y}{y_0}\right)^{p^-/2}.
\]
Therefore,
\[
\frac{\dd}{\dd t}y^{-(p^--2)/2}
\le(p^--2)p^+\mathcal J(u_0)y_0^{-p^-/2},
\]
and integration yields
\[
y(t)^{-(p^--2)/2}
\le y_0^{-(p^--2)/2}
\bigl(1+(p^--2)\Lambda_0(u_0)t\bigr).
\]
If $T_*<\infty$, this estimate bounds $y$ away from zero on $[0,T_*)$, contradicting continuity and the definition of $T_*$. Thus $T_*=\infty$, and solving for $y$ gives~\eqref{eq:lower-bound}. The estimate for $\Lambda_0$ follows from~\eqref{eq:J-rho-comparison}.
\end{proof}

\begin{remark}\label{rem:variable-example}
Let
\[
p(x)=p_0+\varepsilon\cos(2\pi x_1),\qquad
2<p_0-\varepsilon<p_0+\varepsilon<p_{\mathrm c}.
\]
Then Assumption~\ref{ass:p} holds with $p^-=p_0-\varepsilon$ and $p^+=p_0+\varepsilon$. For nonzero $u_0\in V$ with $\|u_0\|_{L^2}\le R$, Corollary~\ref{cor:baseline-decay} and Theorem~\ref{thm:lower-bound} give
\[
\frac{\|u_0\|_{L^2}}
{(1+(p_0-\varepsilon-2)\Lambda_0(u_0)t)^{1/(p_0-\varepsilon-2)}}
\le\|u^{\mathrm{det}}(t)\|_{L^2}
\le
\frac{\|u_0\|_{L^2}}
{(1+c_R(p_0+\varepsilon-2)t\|u_0\|_{L^2}^{p_0+\varepsilon-2})^{1/(p_0+\varepsilon-2)}}.
\]

The two powers differ when $\varepsilon>0$. Theorem~\ref{thm:exponential-main} gives any prescribed exponential rate, with a noise choice uniform on $\|u_0\|_{L^2}\le R$ and frequency growth quantified in Remark~\ref{rem:frequency-growth}.
\end{remark}

\subsection{Iteration inequalities and heat semigroup estimates}

\begin{lemma}\label{lem:iteration}
Let $m>2$ and $h=(m-2)/2$. For $x\ge0$,
\begin{equation}\label{eq:tangent}
1-\frac{x}{h}\le(1+x)^{-1/h}.
\end{equation}
If $a,b,x,y,z\ge0$ satisfy $y\le x/(1+ax^h)^{1/h}$ and
$z\le y/(1+by^h)^{1/h}$, then
\begin{equation}\label{eq:iteration}
z\le\frac{x}{(1+(a+b)x^h)^{1/h}}.
\end{equation}
\end{lemma}
\begin{proof}
Convexity gives~\eqref{eq:tangent}. For positive $x,y,z$, the hypotheses yield $z^{-h}\ge y^{-h}+b\ge x^{-h}+a+b$, which is~\eqref{eq:iteration}. The zero cases are immediate.
\end{proof}

Let $S(t):=e^{t\Delta}$ be the heat semigroup on $\mathbb T^d$.  Write
\[
e_k(x):=e^{2\pi i k\cdot x},
\qquad
\lambda_k:=4\pi^2\abs{k}^2,
\qquad k\in\mathbb Z^d.
\]
For complex Fourier modes, we employ the canonical complexifications of \(H\), \(V\), and \(V^*\), where
\[
\pair{f}{g}_{L^2}:=\int_{\mathbb T^d}f\overline g\,\mathrm{d}x
\]
is linear in the first variable. All estimates then restrict to the original real spaces. On the zero-mean subspace, the spectral gap equals \(\lambda_1=4\pi^2\). 
Moreover, under Assumption~\ref{ass:p}, periodic heat-kernel convolution is uniformly bounded on \(L^{p(\cdot)}\) for bounded positive times; see~\cite{DieningBook}. 

\begin{lemma}\label{lem:heat-C0-V}
The heat semigroup $S(t)=e^{t\Delta}$ leaves $V$ invariant and is strongly continuous on $V$.  Its adjoint realization on $V^*$ is also strongly continuous.  Consequently, for every $T>0$,
\begin{equation}\label{eq:heat-Vstar-bounded}
\sup_{0\le t\le T}\norm{S(t)}_{\mathcal L(V^*)}<\infty.
\end{equation}
\end{lemma}

\begin{proof}
Uniform boundedness of periodic heat convolution on $L^{p(\cdot)}$, the identity $\nabla S(t)v=S(t)\nabla v$, and smooth density give strong continuity on $V$. The dual realization
\[
\langle S_{V^*}(t)f,v\rangle:=\langle f,S_V(t)v\rangle
\]
has the same local bound and agrees with $S(t)$ on $H$. Since $H$ is dense in $V^*$, approximation gives strong continuity on $V^*$ and~\eqref{eq:heat-Vstar-bounded}.
\end{proof}

\begin{lemma}\label{lem:heat-variable}
Let $\eta$ be defined by~\eqref{eq:eta-def-intro}.  There exists a constant
$\hat{C}_h=\hat{C}_h(d,p^-,p^+)>0$, independent of $s$, such that for every $\phi\in H$,
\begin{equation*}
\norm{\nabla S(s)\phi}_{L^{p(\cdot)}}
\le \hat{C}_hs^{-\eta}\norm{\phi}_{L^2},
\qquad s>0.
\end{equation*}
Equivalently, $\|S(s)\phi\|_V\le \hat{C}_hs^{-\eta}\|\phi\|_{L^2}$.
\end{lemma}

\begin{proof}
For $0<s\le1$, the embedding $L^{p^+}\hookrightarrow L^{p(\cdot)}$ and the heat estimate~\cite{Lunardi1995} give
\[
\|\nabla S(s)\phi\|_{L^{p(\cdot)}}
\le C\|\nabla S(s)\phi\|_{L^{p^+}}
\le Cs^{-\eta}\|\phi\|_{L^2}.
\]
For $s\ge1$, the semigroup property and the spectral gap yield
\[
\|\nabla S(s)\phi\|_{L^{p(\cdot)}}
\le C\|S(s-1/2)\phi\|_{L^2}
\le Ce^{-\lambda_1(s-1/2)}\|\phi\|_{L^2}
\le Cs^{-\eta}\|\phi\|_{L^2}.
\]
\end{proof}

\begin{corollary}\label{cor:heat-Vstar-H}
There exists $C>0$ such that for every $f\in V^*$ and every $s>0$,
\begin{equation*}
\norm{S(s)f}_H\le Cs^{-\eta}\norm{f}_{V^*}.
\end{equation*}
Here $S(s)$ on $V^*$ is the adjoint realization from Lemma~\ref{lem:heat-C0-V}.
\end{corollary}

\begin{proof}
Lemma~\ref{lem:heat-variable} gives
$|\langle f,S(s)\phi\rangle|\le Cs^{-\eta}\|f\|_{V^*}\|\phi\|_H$
for every $\phi\in H$. By the Riesz representation theorem, this functional is identified with an element of \(H\) that coincides with the dual realization of \(S(s)f\) on \(V\).
\end{proof}

\section{Transport noise and variational solutions}\label{sec:noise}

\subsection{Fourier construction}

Let $\mathbb Z_0^d:=\mathbb Z^d\setminus\{0\}$ and $C_d:=d/(d-1)$. For each $k\in\mathbb Z_0^d$, choose an orthonormal basis $\{a_{k,i}\}_{i=1}^{d-1}$ of $k^\perp$ with $a_{-k,i}=a_{k,i}$, and set $\sigma_{k,i}=a_{k,i}e_k$. Then $\Div\sigma_{k,i}=0$ and $\sigma_{k,i}\cdot\nabla e_{\pm k}=0$. Let $\theta\in\ell^2(\mathbb Z_0^d;\mathbb R)$ satisfy
\begin{equation}\label{eq:theta-assumptions}
\theta_k=\theta_l\quad\text{if }|k|=|l|,
\qquad
\sum_{k\in\mathbb Z_0^d}\theta_k^2=1.
\end{equation}
Thus $\theta$ is invariant under signed coordinate permutations and $\theta_{-k}=\theta_k$.

Let $(\Omega,\mathcal F,\{\mathcal F_t\}_{t\ge0},\mathbb P)$ be a filtered probability space satisfying the usual conditions, with $\mathcal F$ countably generated modulo null sets. 
Let $\{W_t^{k,i}\}_{k\in\mathbb Z_0^d,1\le i\le d-1}$ be standard complex Brownian motions defined on $(\Omega,\mathcal F,\{\mathcal F_t\}_{t\ge0},\mathbb P)$,
satisfying
\begin{equation*}
W_t^{-k,i}=\overline{W_t^{k,i}},
\qquad
[W^{k,i},W^{l,j}]_t=2t\delta_{k,-l}\delta_{ij}.
\end{equation*}
For an intensity $\kappa>0$, set
\begin{equation*}
W(t,x):=\sqrt{C_d\kappa}
\sum_{k\in\mathbb Z_0^d}\sum_{i=1}^{d-1}
\theta_k\sigma_{k,i}(x)W_t^{k,i}.
\end{equation*}
The symmetry relations imply that $W$ is real valued.

Choose a partition $\mathbb Z_0^d=\mathbb Z_+^d\cup\mathbb Z_-^d$ with $\mathbb Z_-^d=-\mathbb Z_+^d$.  Define real vector fields
\begin{equation}\label{eq:xi-def}
\xi_{k,i}(x):=
\begin{cases}
2\sqrt{C_d\kappa}\,\theta_k\operatorname{Re}\sigma_{k,i}(x),&k\in\mathbb Z_+^d,\\
2\sqrt{C_d\kappa}\,\theta_k\operatorname{Im}\sigma_{k,i}(x),&k\in\mathbb Z_-^d,
\end{cases}
\end{equation}
and independent standard real Brownian motions
\begin{equation*}
B_t^{k,i}:=
\begin{cases}
\operatorname{Re}W_t^{k,i},&k\in\mathbb Z_+^d,\\
-\operatorname{Im}W_t^{k,i},&k\in\mathbb Z_-^d.
\end{cases}
\end{equation*}
Then $W(t,x)=\sum_{k,i}\xi_{k,i}(x)B_t^{k,i}$. Let $\mathbb W$ be the cylindrical Wiener process on
$\mathcal U:=\ell^2(\mathbb Z_0^d\times\{1,\ldots,d-1\})$
with coordinates $B^{k,i}$ and canonical basis $\{\mathbf e_{k,i}\}$. The transport term is $B(u)\dd\mathbb W_t$, where $B(u)\mathbf e_{k,i}=\xi_{k,i}\cdot\nabla u$. Each field is smooth and divergence-free, and
\begin{equation}\label{eq:xi-linfty-sum}
\sum_{k,i}\norm{\xi_{k,i}}_{L^\infty}^2
\le4C_d\kappa(d-1)\sum_k\theta_k^2=4d\kappa.
\end{equation}

\subsection{The Stratonovich-It\^{o} corrector}

\begin{lemma}\label{lem:isotropy}
The Fourier profile~\eqref{eq:theta-assumptions} satisfies
\begin{equation}\label{eq:isotropy}
\sum_{k\in\mathbb Z_0^d}\theta_k^2
\sum_{i=1}^{d-1}a_{k,i}\otimes a_{k,i}
=\frac{d-1}{d}I_d
=\frac1{C_d}I_d.
\end{equation}
\end{lemma}

\begin{proof}
The matrix $M=\sum_k\theta_k^2(k/|k|)\otimes(k/|k|)$ is invariant under signed coordinate permutations and has trace one, so $M=I_d/d$. Since $\sum_i a_{k,i}\otimes a_{k,i}=I_d-(k/|k|)\otimes(k/|k|)$, we obtain~\eqref{eq:isotropy}.
\end{proof}

\begin{proposition}\label{prop:correction}
For every smooth periodic function $u$,
\begin{equation*}
\frac12\sum_{k,i}\xi_{k,i}\cdot\nabla
\bigl(\xi_{k,i}\cdot\nabla u\bigr)
=\kappa\Delta u.
\end{equation*}
\end{proposition}

\begin{proof}
Since $a_{k,i}\cdot k=0$, we have $\xi_{k,i}\cdot\nabla\xi_{k,i}=0$. Pairing sine and cosine fields and using~\eqref{eq:isotropy}, we get
\begin{equation}\label{eq:covariance-xi}
\sum_{k,i}\xi_{k,i}(x)\otimes\xi_{k,i}(x)
=2C_d\kappa\sum_k\theta_k^2\sum_i a_{k,i}\otimes a_{k,i}
=2\kappa I_d.
\end{equation}
The sum converges absolutely and uniformly by~\eqref{eq:xi-linfty-sum}. Contracting with $D^2u/2$ proves the assertion.
\end{proof}

\begin{corollary}\label{cor:finite-annular}
For $N\in\mathbb N$, set
\[
\mathcal K_N:=\{k\in\mathbb Z_0^d:N\le|k|\le2N\},\qquad
\theta_k^{(N)}:=|\mathcal K_N|^{-1/2}\mathbf1_{\mathcal K_N}(k).
\]
These profiles satisfy~\eqref{eq:theta-assumptions}, have finite support, and retain the covariance $2\kappa I_d$ and correction $\kappa\Delta$ of Proposition~\ref{prop:correction}. Moreover,
\begin{equation*}
\|\theta^{(N)}\|_{\ell^\infty}=|\mathcal K_N|^{-1/2}\longrightarrow0.
\end{equation*}
\end{corollary}

\begin{proof}
The profiles are radial and normalized, so Proposition~\ref{prop:correction} applies. Since $|\mathcal K_N|\to\infty$, we also have $\|\theta^{(N)}\|_{\ell^\infty}\to0$.
\end{proof}

\begin{corollary}\label{cor:variational-correction}
For $u,v\in V$, define the transport correction by
\begin{equation}\label{eq:Cxi-variational}
\pair{\mathcal C_\xi(u)}{v}_{V^*,V}
:=-\frac12\sum_{k,i}\int_{\mathbb T^d}
(\xi_{k,i}\cdot\nabla u)(\xi_{k,i}\cdot\nabla v)\dd x.
\end{equation}
Then the series is absolutely convergent and
\begin{equation}\label{eq:Cxi-kappaDelta}
\pair{\mathcal C_\xi(u)}{v}_{V^*,V}
=-\kappa\int_{\mathbb T^d}\nabla u\cdot\nabla v\dd x
=\pair{\kappa\Delta u}{v}_{V^*,V}.
\end{equation}
\end{corollary}

\begin{proof}
Since $\nabla u,\nabla v\in L^2$,~\eqref{eq:xi-linfty-sum} gives
\[
\sum_{k,i}\int_{\mathbb T^d}|(\xi_{k,i}\cdot\nabla u)(\xi_{k,i}\cdot\nabla v)|\dd x
\le4d\kappa\|\nabla u\|_{L^2}\|\nabla v\|_{L^2}.
\]
We may therefore interchange the sum and integral and apply~\eqref{eq:covariance-xi} to obtain~\eqref{eq:Cxi-kappaDelta}.
\end{proof}

Motivated by Proposition~\ref{prop:correction} and Corollary~\ref{cor:variational-correction}, we rigorously define the equation by the It\^o variational problem
\begin{equation}\label{eq:ito-main}
\dd u
=\Dpx u \dd t+\kappa\Delta u\dd t
+\sum_{k,i}\xi_{k,i}\cdot\nabla u \dd B_t^{k,i}.
\end{equation}
\subsection{Variational well-posedness and the energy identity}\label{sec:wellposed-main}

In this subsection, assume only that $p$ is periodic and log-H\"older continuous with $2<p^-\le p^+<\infty$. The restriction $p^+<p_{\mathrm c}$ enters in Subsection~\ref{sec:mild}. We first obtain a weak limit using the compatible correction~\eqref{eq:CN-def}, then apply Proposition~\ref{prop:variable-Ito} to its unidentified flux, and finally use Minty's argument.
For \(T>0\), set
\[
Q_T:=\Omega\times(0,T)\times\mathbb T^d,
\qquad
\widetilde p(\omega,t,x):=p(x).
\]
When convenient, we relabel the double index \((k,i)\) by a single index \(j\) and write
\[
B(u)\mathbf e_j:=\xi_j\cdot\nabla u,
\qquad
\pair{\mathcal C_\xi(u)}{v}_{V^*,V}
:=-\frac12\sum_j\int_{\mathbb T^d}(\xi_j\cdot\nabla u)(\xi_j\cdot\nabla v)\dd x.
\]
For the Fourier fields of Section~\ref{sec:noise},
Corollary~\ref{cor:variational-correction}
yields \(\mathcal C_\xi(u)=\kappa\Delta u\).

\begin{definition}\label{def:variational}
An adapted process \(u\) is called a variational solution
of~\eqref{eq:ito-main} on \([0,T]\) if \(u\) has weakly continuous \(H\)-valued paths,
\[
\mathbb E\sup_{0\le t\le T}\norm{u(t)}_H^2<\infty,
\qquad
\nabla u\in L^{\widetilde p}(Q_T;\mathbb R^d),
\]
\(u\) and \(\nabla u\) admit predictable representatives, and,
on one set of probability one, for every \(t\in[0,T]\) and \(v\in V\),
\begin{align}
\pair{u(t)}{v}
={}&\pair{u_0}{v}
-\int_0^t\int_{\mathbb T^d}
\abs{\nabla u}^{p(x)-2}\nabla u\cdot\nabla v\dd x\dd s\notag\\
&-\kappa\int_0^t\int_{\mathbb T^d}\nabla u\cdot\nabla v\dd x\dd s
+\sum_j\int_0^t\pair{\xi_j\cdot\nabla u}{v}\dd B_s^j.
\label{eq:variational-main}
\end{align}
\end{definition}

 If $z=u-v$, then
\begin{equation*}
2\pair{\kappa\Delta z}{z}_{V^*,V}
+\norm{B(u)-B(v)}_{\mathcal L_2(\mathcal U,H)}^2=0,
\end{equation*}
because $\sum_j\xi_j\otimes\xi_j=2\kappa I_d$.  Together with the monotonicity of the $p(x)$-Laplace operator, this yields
\begin{equation}\label{eq:main-weak-monotonicity}
2\pair{\Dpx u-\Dpx v+\kappa\Delta z}{z}
+\norm{B(u)-B(v)}_{\mathcal L_2(\mathcal U,H)}^2\le0.
\end{equation}
In particular, the transport correction contributes no net $L^2$ energy.

\subsubsection{Galerkin approximation and energy identity}

Let \(H_N\) denote the real zero-mean trigonometric space generated by the frequencies \(0<|k|\le N\), and let $P_N:H\to H_N$ be the $L^2$-orthogonal projection.  It is self-adjoint, commutes with derivatives in the $L^2$ sense, and, for every smooth mean-zero $\varphi$,
\begin{equation}\label{eq:PN-smooth-convergence}
P_N\varphi\longrightarrow\varphi
\quad\text{in }W^{r,q}(\mathbb T^d)
\quad\text{for every }r\ge0\text{ and }1<q<\infty.
\end{equation}
Define $A_N(u)\in H_N$ weakly by
\begin{equation}\label{eq:AN-def}
\pair{A_N(u)}{v}_H
=-\int_{\mathbb T^d}|\nabla u|^{p(x)-2}\nabla u\cdot\nabla v\dd x,
\qquad v\in H_N,
\end{equation}
and employ the compatible correction
\begin{equation}\label{eq:CN-def}
\mathcal C_N(u):=\frac12\sum_jP_NL_jP_NL_ju,
\qquad L_ju:=\xi_j\cdot\nabla u.
\end{equation}
This sum is well defined on $H_N$.  Indeed, if the field $\xi_j$ corresponds to a Fourier frequency $k$ and $u\in H_N$, then $P_NL_ju=0$ whenever $|k|>2N$.
Hence only finitely many lattice frequencies contribute for each
fixed \(N\), even when the original profile has infinite support. The finite-dimensional equation
\begin{equation}\label{eq:Galerkin}
\dd u_N=A_N(u_N)\dd t+\mathcal C_N(u_N)\dd t
+\sum_jP_NL_ju_N\dd B_t^j,
\qquad u_N(0)=P_Nu_0,
\end{equation}
has locally Lipschitz coefficients. Indeed, $H_N$ is finite dimensional and $z\mapsto|z|^{p(x)-2}z$ is locally Lipschitz uniformly for $p(x)\in[p^-,p^+]$. Hence it has a
unique local solution. Applying It\^o's formula up to its exit time, the skew-adjointness of \(L_j\) and the self-adjointness of \(P_N\) yield
\begin{equation}\label{eq:main-Galerkin-energy}
\norm{u_N(t)}_H^2
+2\int_0^t\rho_{p(\cdot)}(\nabla u_N(s))\dd s
=\norm{P_Nu_0}_H^2.
\end{equation}
In fact,
\[
2\pair{\mathcal C_N(u_N)}{u_N}
=-\sum_j\norm{P_NL_ju_N}_{L^2}^2,
\qquad
\pair{u_N}{P_NL_ju_N}=0,
\]
so the correction, quadratic variation and scalar martingale cancel exactly.  Consequently, for every $T>0$,
\begin{equation}\label{eq:Galerkin-uniform}
\begin{aligned}
&\sup_N\mathbb E\sup_{0\le t\le T}\norm{u_N(t)}_H^2
\le\mathbb E\norm{u_0}_H^2,\\
2&\sup_N\mathbb E\int_0^T\rho_{p(\cdot)}(\nabla u_N(s))\dd s
\le\mathbb E\norm{u_0}_H^2,
\end{aligned}
\end{equation}
so no finite-dimensional blow-up occurs and~\eqref{eq:Galerkin} has a global solution.

\subsubsection{Compactness and the unidentified limit equation}\label{subsec:galerkin-limit}

\begin{lemma}\label{lem:predictable-closed}
Let $1<r^-\le r(x)\le r^+<\infty$ and $\widetilde r(\omega,t,x)=r(x)$. The subspace of $L^{\widetilde r}(Q_T;\mathbb R^N)$ admitting predictable representatives is norm closed and weakly closed. The same holds in $L^2(\Omega\times(0,T);X)$ for a separable Hilbert space $X$. Here predictability is understood with respect to the completed product of the predictable $\sigma$-field and $\mathcal B(\mathbb T^d)$.
\end{lemma}
\begin{proof}
Norm convergence has an almost-everywhere convergent subsequence, whose limit is predictable. The subspace is therefore norm closed and, being linear, weakly closed. The same argument applies to $X$-valued functions.
\end{proof}

\begin{lemma}\label{lem:weak-H-representative}
Let $V\hookrightarrow H\hookrightarrow V^*$ be the dense continuous Gelfand triple used above.  If
\[
v\in C([0,T];V^*)\cap L^\infty(0,T;H),
\]
then $v$ has a unique representative in $C_w([0,T];H)$ and
\begin{equation}\label{eq:weak-H-bound}
\sup_{t\in[0,T]}\norm{v(t)}_H
\le\operatorname*{ess\,sup}_{t\in(0,T)}\norm{v(t)}_H.
\end{equation}
\end{lemma}
\begin{proof}
Choose \(t_n\to t\) from the full-measure set on which the \(H\)
bound holds. Any weak \(H\) limit of \(v(t_n)\) coincides with the \(V^*\) limit \(v(t)\), and weak lower semicontinuity
yields~\eqref{eq:weak-H-bound}. Approximating \(h\in H\) by elements of \(V\) gives continuity of \(t\mapsto\pair{v(t)}{h}_H\);
see~\cite{PrevotRockner2007,LiuRockner2015}.
\end{proof}

By~\eqref{eq:Galerkin-uniform}, the gradients are bounded in $L^{\widetilde p}(Q_T)$. Since $p^->2$ and $Q_T$ has measure $T$,
\begin{equation}\label{eq:grad-L2-product}
\sup_N\mathbb{E}\int_0^T\norm{\nabla u_N(t)}_{L^2}^2\dd t
\le T+\sup_N\mathbb{E}\int_0^T\rho_{p(\cdot)}(\nabla u_N(t))\dd t<\infty.
\end{equation}
For $F_N:=J_{p(\cdot)}(u_N)$,~\eqref{eq:flux-modular} and~\eqref{eq:main-Galerkin-energy} give
\begin{align*}
\int_{Q_T}|F_N(\omega,t,x)|^{p'(x)}\dd\mathbb P\dd t\dd x
&=\mathbb{E}\int_0^T\rho_{p(\cdot)}(\nabla u_N(t))\dd t\\
&\le\frac12\mathbb{E}\norm{P_Nu_0}_H^2
\le\frac12\mathbb{E}\norm{u_0}_H^2.
\end{align*}
By reflexivity and weak-$*$ compactness, we may choose a subsequence such that
\begin{align}
u_N&\rightharpoonup u
\text{ weakly-\(*\) in }L^\infty\bigl(0,T;L^2(\Omega;H)\bigr),
\notag\\
u_N&\rightharpoonup u
\text{ weakly in }L^2\bigl(\Omega\times(0,T);H\bigr),
\label{eq:weakL2-uN}\\
\nabla u_N&\rightharpoonup Y
\text{ weakly in }L^{\widetilde p}(Q_T;\mathbb R^d),
\label{eq:weak-grad-uN}\\
F_N&\rightharpoonup F
\text{ weakly in }L^{\widetilde p'}(Q_T;\mathbb R^d).
\label{eq:weak-FN}
\end{align}
For smooth vector fields $\phi_r$ and bounded predictable $Z_r$, set $\psi=\sum_{r=1}^M Z_r\phi_r$. By integration by parts and~\eqref{eq:weakL2-uN}--\eqref{eq:weak-grad-uN}, we have
\[
\int_{Q_T}Y\cdot\psi=-\int_{Q_T}u\,\Div_x\psi.
\]
By density, $Y=\nabla u$ distributionally, and
\begin{equation}\label{eq:main-weak-grad}
\nabla u_N\rightharpoonup\nabla u
\quad\text{in }L^{\widetilde p}(Q_T;\mathbb R^d).
\end{equation}
Lemma~\ref{lem:predictable-closed} gives predictable representatives of $u$, $\nabla u$ and $F$.

Define
\[
B_N(t)\mathbf e_j:=P_NL_ju_N(t),
\qquad
B(u(t))\mathbf e_j:=L_ju(t),
\qquad L_jv:=\xi_j\cdot\nabla v.
\]

By~\eqref{eq:grad-L2-product} and~\eqref{eq:xi-linfty-sum}, $B_N$ is bounded in $L^2_{\mathcal P}(\Omega\times(0,T);\mathcal L_2(\mathcal U,H))$. The transport map is linear and satisfies $\|B(v)\|_{\mathcal L_2}^2\le4d\kappa\|\nabla v\|_{L^2}^2$. Thus the weak $L^2$ convergence of the gradients supplied by~\eqref{eq:main-weak-grad} gives weak convergence of $B(u_N)$ to $B(u)$. For every square-integrable Hilbert--Schmidt test $\Phi$, strong convergence of $P_N$ and dominated convergence give $P_N\Phi\to\Phi$ in the same $L^2$ space. Since $B_N=P_NB(u_N)$, we obtain
\begin{equation*}
B_N\rightharpoonup B(u)
\quad\text{in }L^2_{\mathcal P}(\Omega\times(0,T);\mathcal L_2(\mathcal U,H)).
\end{equation*}
The bounded linear It\^o map yields
\begin{equation}\label{eq:stochastic-integral-weak}
\int_0^tB_N(r)\dd\mathbb W_r
\rightharpoonup
\int_0^tB(u(r))\dd\mathbb W_r
\quad\text{in }L^2(\Omega;H)
\end{equation}
for every $t\in[0,T]$. The space-time It\^o map is also bounded by the isometry and Fubini's theorem. With
\[
\mathbb M_N(t):=\int_0^tB_N(r)\dd\mathbb W_r,
\qquad
\mathbb M(t):=\int_0^tB(u(r))\dd\mathbb W_r,
\]
we therefore obtain
\begin{equation}\label{eq:stochastic-integral-weak-spacetime}
\mathbb M_N\rightharpoonup\mathbb M
\quad\text{in }L^2(\Omega\times(0,T);H).
\end{equation}

For $v\in V$, by the $L^2$ contractivity of $P_N$ we have $\|\nabla P_Nv\|_{L^2}\le\|\nabla v\|_{L^2}$. Hence
\begin{align*}
\abs{\pair{\mathcal C_N(u_N)}{v}}
&\le \frac12
\left(\sum_j\norm{P_NL_ju_N}_{L^2}^2\right)^{1/2}
\left(\sum_j\norm{L_jP_Nv}_{L^2}^2\right)^{1/2}\\
&\le C \kappa\norm{\nabla u_N}_{L^2}\norm{v}_V.
\end{align*}
Consequently, by~\eqref{eq:grad-L2-product} we get
\begin{equation}\label{eq:CN-L2-Vstar}
\sup_N\mathbb{E}\int_0^T\norm{\mathcal C_N(u_N(t))}_{V^*}^2\dd t<\infty.
\end{equation}
The same estimate holds for $\mathcal C_\xi(u)$.  Now let $\varphi$ be smooth and of mean zero. From~\eqref{eq:CN-def}, we have
\begin{equation*}
\pair{\mathcal C_N(u_N)}{\varphi}
=-\frac12\sum_{j=1}^\infty
\pair{P_NL_ju_N}{L_jP_N\varphi}_H.
\end{equation*}
For fixed $J$, the sum over $j\le J$ converges by~\eqref{eq:weakL2-uN} and~\eqref{eq:PN-smooth-convergence}, while the remaining tail satisfies
\begin{align*}
&~~\sum_{j>J}\abs{\pair{P_NL_ju_N}{L_jP_N\varphi}}\\
&\le
\left(\sum_{j>J}\norm{\xi_j\cdot\nabla u_N}_{L^2}^2\right)^{1/2}
\left(\sum_{j>J}\norm{\xi_j\cdot\nabla P_N\varphi}_{L^2}^2\right)^{1/2}\\
&\le
\left(\sum_{j>J}\norm{\xi_j}_{L^\infty}^2\right)
\norm{\nabla u_N}_{L^2}\norm{\nabla P_N\varphi}_{L^2}.
\end{align*}
The last expression tends to zero in $L^1(\Omega\times(0,T))$ uniformly in $N$ as $J\to\infty$. Consequently, for every smooth mean-zero $\varphi$ and every bounded measurable scalar process $\zeta$, it holds
\begin{equation}\label{eq:correction-limit}
\mathbb{E}\int_0^T\zeta(t)\pair{\mathcal C_N(u_N(t))}{\varphi}\dd t
\longrightarrow
\mathbb{E}\int_0^T\zeta(t)\pair{\mathcal C_\xi(u(t))}{\varphi}\dd t.
\end{equation}
Smooth spatial tests with bounded measurable coefficients are dense in $L^2(\Omega\times(0,T);V)$. Hence~\eqref{eq:CN-L2-Vstar} and~\eqref{eq:correction-limit} imply
\begin{equation}\label{eq:correction-weak-Vstar}
\mathcal C_N(u_N)\rightharpoonup \mathcal C_\xi(u)
\quad\text{in }L^2(\Omega\times(0,T);V^*).
\end{equation}
For smooth mean-zero $\varphi$ and bounded measurable $\zeta$,~\eqref{eq:AN-def} similarly gives
\begin{align}
&~~\mathbb E\int_0^T\zeta(t)
\pair{A_N(u_N(t))}{P_N\varphi}_H\dd t\notag\\
&=-\mathbb E\int_0^T\zeta(t)
\int_{\mathbb T^d}F_N(t,x)\cdot\nabla P_N\varphi(x)\dd x\dd t\notag\\
&\longrightarrow
-\mathbb E\int_0^T\zeta(t)
\int_{\mathbb T^d}F(t,x)\cdot\nabla\varphi(x)\dd x\dd t.
\label{eq:nonlinear-drift-weak-integrated}
\end{align}
Here $\nabla P_N\varphi\to\nabla\varphi$ strongly in $L^{p(\cdot)}$ and $F_N\rightharpoonup F$ weakly in $L^{\widetilde p'}$.

Let
\begin{equation*}
\widetilde u(t)
:=u_0
+\int_0^t\Div F(r)\dd r
+\int_0^t\mathcal C_\xi(u(r))\dd r
+\mathbb M(t),
\qquad 0\le t\le T.
\end{equation*}
Variable-exponent H\"older's inequality and~\eqref{eq:CN-L2-Vstar} make the drift integrals absolutely continuous in $V^*$. Since $\mathbb M$ is $H$-continuous, we have
\[
\widetilde u\in C([0,T];V^*)
\qquad\text{almost surely}.
\]

For smooth mean-zero $\varphi$, $\eta\in C_c^\infty(0,T)$ and bounded random $Z$, test the Galerkin equation by $P_N\varphi$ and integrate against $Z\eta(t)$. By~\eqref{eq:weakL2-uN},~\eqref{eq:weak-FN},~\eqref{eq:correction-weak-Vstar} and~\eqref{eq:stochastic-integral-weak-spacetime}, we obtain
\[
\mathbb E\int_0^T Z\eta(t)\langle u(t)-\widetilde u(t),\varphi\rangle\dd t=0.
\]
Arbitrariness of $Z$ and countable dense families of $\eta,\varphi$ give $u=\widetilde u$ in $V^*$ for almost every $(\omega,t)$.

Replace $u$ by this $V^*$-continuous representative. On one full-probability set, for every $t\in[0,T]$ and $\varphi\in V$, it holds
\begin{align}
\pair{u(t)}{\varphi}
={}&\pair{u_0}{\varphi}
-\int_0^t\int_{\mathbb T^d}F(r)\cdot\nabla\varphi\dd x\dd r
+\int_0^t\pair{\mathcal C_\xi(u(r))}{\varphi}\dd r\notag\\
&+\sum_{j=1}^\infty\int_0^t\pair{L_ju(r)}{\varphi}\dd B_r^j.
\label{eq:limit-unidentified}
\end{align}

By Mazur's lemma, tail convex combinations $\bar u_N$ converge strongly to $u$ in $L^2(\Omega\times(0,T);H)$. The Galerkin identity gives $\|\bar u_N(t,\omega)\|_H\le\|u_0(\omega)\|_H$. An almost-everywhere convergent subsequence then yields
\[
u(\cdot,\omega)\in L^\infty(0,T;H),
\qquad
\operatorname*{ess\,sup}_{t\in(0,T)}\norm{u(t,\omega)}_H
\le\norm{u_0(\omega)}_H.
\]
Lemma~\ref{lem:weak-H-representative} gives a $C_w([0,T];H)$ representative and extends the bound to every time. Thus
\begin{equation}\label{eq:main-limit-H-bound}
\mathbb{E}\sup_{t\in[0,T]}\norm{u(t)}_H^2
\le\mathbb{E}\norm{u_0}_H^2.
\end{equation}

For fixed $t$, test~\eqref{eq:Galerkin} with $P_N\varphi$ and multiply by a bounded random variable $Z$. Using~\eqref{eq:nonlinear-drift-weak-integrated} with $\zeta=Z\mathbf1_{[0,t]}$,~\eqref{eq:correction-limit}, and~\eqref{eq:stochastic-integral-weak}, we obtain
\[
\mathbb E[Z\langle {u_N(t)},{\varphi}\rangle]
\longrightarrow\mathbb E[Z\langle {u(t)},{\varphi}\rangle].
\]

The uniform $L^2(\Omega;H)$ bound, truncation in $Z$ and spatial density give
\begin{equation}\label{eq:fixed-time-weak}
u_N(t)\rightharpoonup u(t)
\quad\text{in }L^2(\Omega;H)
\quad\text{for every }t\in[0,T].
\end{equation}

\subsubsection{Energy identity before flux identification}

\begin{theorem}\label{thm:pre-Minty-main}
Let $(u,F)$ be the limit pair constructed in Subsection~\ref{subsec:galerkin-limit}. Then $u$ admits an $H$-continuous modification and, almost surely, for every $t\in[0,T]$,
\begin{equation}\label{eq:pre-Minty-energy-main}
\norm{u(t)}_H^2
+2\int_0^t\int_{\mathbb T^d}F(s,x)\cdot\nabla u(s,x)\dd x\dd s
=\norm{u_0}_H^2.
\end{equation}
\end{theorem}

\begin{proof}
The bounds in Subsection~\ref{subsec:galerkin-limit} verify the hypotheses of Proposition~\ref{prop:variable-Ito}. Applying it to~\eqref{eq:limit-unidentified} gives~\eqref{eq:pre-Minty-energy-main} before identifying $F$.
\end{proof}

\subsubsection{Flux identification and global well-posedness}

For predictable $\zeta\in L^{\widetilde p}(Q_T;\mathbb R^d)$, monotonicity gives

\begin{equation}\label{eq:Minty-main-start}
\int_{Q_T}
\left(F_N-|\zeta|^{p(x)-2}\zeta\right)
\cdot(\nabla u_N-\zeta)\dd\mathbb P\dd t\dd x\ge0.
\end{equation}
By~\eqref{eq:main-Galerkin-energy},~\eqref{eq:fixed-time-weak}, and Theorem~\ref{thm:pre-Minty-main}, we have
\begin{align*}
\limsup_{N\to\infty}\mathbb E\int_0^T\!\int_{\mathbb T^d}F_N\cdot\nabla u_N\dd x\dd t
&\le\frac12\mathbb E\bigl(\norm{u_0}_H^2-\norm{u(T)}_H^2\bigr)\\
&=\mathbb E\int_0^T\!\int_{\mathbb T^d}F\cdot\nabla u\dd x\dd t.
\end{align*}
The other terms in~\eqref{eq:Minty-main-start} pass to the limit by~\eqref{eq:weak-FN} and~\eqref{eq:main-weak-grad}, since $\zeta\in L^{\widetilde p}$ and $|\zeta|^{p(x)-2}\zeta\in L^{\widetilde p'}$. We obtain
\[
\int_{Q_T}
\left(F-|\zeta|^{p(x)-2}\zeta\right)
\cdot(\nabla u-\zeta)\dd\mathbb P\dd t\dd x\ge0.
\]
For predictable $\Psi\in L^{\widetilde p}(Q_T;\mathbb R^d)$, set $\zeta_\varepsilon=\nabla u-\varepsilon\Psi$, $J(x,z)=|z|^{p(x)-2}z$ and $q(x)=p(x)/(p(x)-2)$. The estimate
\[
|J(x,a)-J(x,b)|\le C(|a|+|b|)^{p(x)-2}|a-b|
\]
and the identity $1/p'(x)=1/q(x)+1/p(x)$ give, by variable-exponent H\"older,
\[
\|J(\cdot,\zeta_\varepsilon)-J(\cdot,\nabla u)\|_{L^{\widetilde p'}(Q_T)}
\le C\|(|\zeta_\varepsilon|+|\nabla u|)^{p(\cdot)-2}\|_{L^{\widetilde q}(Q_T)}
\|\zeta_\varepsilon-\nabla u\|_{L^{\widetilde p}(Q_T)},
\]
where $\widetilde q(\omega,t,x)=q(x)$. The first term is uniformly bounded because its modular equals $\int_{Q_T}(|\zeta_\varepsilon|+|\nabla u|)^{p(x)}$, while the second term tends to zero. Thus $J(\cdot,\zeta_\varepsilon)\to J(\cdot,\nabla u)$ strongly in $L^{\widetilde p'}$ (see, e.g.,~\cite{DieningBook}). Substituting $\zeta_\varepsilon$ into the above, dividing by $\varepsilon>0$, and letting $\varepsilon\downarrow0$ therefore gives
\[
\int_{Q_T}\left(F-|\nabla u|^{p(x)-2}\nabla u\right)\cdot\Psi
\dd\mathbb P\dd t\dd x\ge0.
\]
Replacing $\Psi$ by $-\Psi$ yields
\begin{equation}\label{eq:main-flux-identification}
F=|\nabla u|^{p(x)-2}\nabla u
\quad\text{a.e. on }Q_T.
\end{equation}
Thus~\eqref{eq:limit-unidentified} is precisely the variational equation~\eqref{eq:variational-main}.

\begin{theorem}\label{thm:wellposed}
Let $p$ be periodic and log-H\"older continuous, with $2<p^-\le p^+<\infty$. Fix $\kappa>0$ and a deterministic profile $\theta$ satisfying~\eqref{eq:theta-assumptions}. For every $u_0\in L^2(\Omega,\mathcal F_0,\mathbb P;H)$, equation~\eqref{eq:ito-main} has a unique variational solution satisfying, for every $T>0$,
\begin{equation*}
u\in L^2\bigl(\Omega;C([0,T];H)\bigr),
\qquad
\mathbb{E}\int_0^T\rho_{p(\cdot)}(\nabla u(s))\dd s<\infty.
\end{equation*}
Moreover, almost surely, for every $t\ge0$,
\begin{equation}\label{eq:energy-identity}
\norm{u(t)}_{L^2}^2
+2\int_0^t\rho_{p(\cdot)}(\nabla u(s))\dd s
=\norm{u_0}_{L^2}^2.
\end{equation}
In particular, $t\mapsto\norm{u(t)}_{L^2}$ is nonincreasing.
\end{theorem}

\begin{proof}
The construction above and~\eqref{eq:main-flux-identification} give existence. Theorem~\ref{thm:pre-Minty-main} gives strong $H$ continuity and, after flux identification,~\eqref{eq:energy-identity}.

For two solutions $u,v$, set $z=u-v$. The gradient and flux differences belong to $L^{\widetilde p}$ and $L^{\widetilde p'}$, respectively, and $B(u)-B(v)=B(z)$ is square integrable by~\eqref{eq:xi-linfty-sum}. Proposition~\ref{prop:variable-Ito} applied to the difference equation gives
\[
\dd\norm{z}_H^2
=\left[
2\pair{\Dpx u-\Dpx v+\kappa\Delta z}{z}
+\norm{B(u)-B(v)}_{\mathcal L_2(\mathcal U,H)}^2
\right]\dd t,
\]
because $\pair{z}{\xi_j\cdot\nabla z}=0$ for every $j$.  The bracket is nonpositive by~\eqref{eq:main-weak-monotonicity}.  Hence
\[
\norm{u(t)-v(t)}_H^2\le\norm{u_0-v_0}_H^2.
\]
This gives pathwise uniqueness. Applying the construction on $[0,N]$, $N\in\mathbb N$, and identifying overlaps yields the global solution and~\eqref{eq:energy-identity} for all $t\ge0$ on one full-probability set.
\end{proof}

\subsection{The mild formulation}\label{sec:mild}

\begin{theorem}\label{thm:mild}
Suppose that Assumption~\ref{ass:p} holds. Let $u$ be the variational solution of~\eqref{eq:ito-main} and let $s\in[0,T]$ be deterministic. On one set of probability one, for every $t\in[s,T]$,
\begin{equation}\label{eq:mild}
\begin{aligned}
u(t)
={}&S\bigl(\kappa(t-s)\bigr)u(s)
+\int_s^t S\bigl(\kappa(t-r)\bigr)\Dpx u(r)\dd r\\
&+\sum_{k,i}\int_s^t S\bigl(\kappa(t-r)\bigr)
\bigl(\xi_{k,i}\cdot\nabla u(r)\bigr)\dd B_r^{k,i}.
\end{aligned}
\end{equation}
The deterministic convolution is initially a Bochner integral in $V^*$ and admits an $H$-continuous representative. The stochastic convolution also has an $H$-continuous modification, so~\eqref{eq:mild} holds in $H$. For any countable family of deterministic starting times, these assertions hold on a common set of probability one.
\end{theorem}

\begin{proof}
Let $D(r)=\rho_{p(\cdot)}(\nabla u(r))$ and use the transport coefficient $B(u(r))$ defined above. By Proposition~\ref{prop:px-properties} and the energy identity, we obtain
\[
\int_s^T\|\Dpx u(r)\|_{V^*}\dd r
\le2\int_s^T(D(r)^{a_-}+D(r)^{a_+})\dd r<\infty
\quad\text{almost surely}.
\]
The boundedness and strong continuity of $S$ on $V^*$ therefore yield a continuous $V^*$-valued deterministic convolution. Dual heat smoothing together with Tonelli's theorem then gives
\begin{align*}
&~~\int_s^T\left\|\int_s^t S(\kappa(t-r))\Dpx u(r)\dd r\right\|_H\dd t\\
&\le C\kappa^{-\eta}\int_s^T\int_s^t(t-r)^{-\eta}
(D(r)^{a_-}+D(r)^{a_+})\dd r\dd t\\
&=\frac{C\kappa^{-\eta}}{1-\eta}\int_s^T(T-r)^{1-\eta}
(D(r)^{a_-}+D(r)^{a_+})\dd r<\infty.
\end{align*}
Thus the deterministic convolution is Bochner integrable in $H$ for almost every $t$.

The process $B(u)$ is predictable, and~\eqref{eq:xi-linfty-sum} gives
\begin{equation}\label{eq:mild-HS}
\mathbb E\int_s^T\|B(u(r))\|_{\mathcal L_2(\mathcal U,H)}^2\dd r
\le4d\kappa\mathbb E\int_s^T\|\nabla u(r)\|_{L^2}^2\dd r<\infty.
\end{equation}
The last inequality follows from $p^->2$ and the modular bound. Since $S$ is a strongly continuous contraction semigroup on $H$,~\cite[Theorem~4.1]{vanNeervenVeraar2022} gives an $H$-continuous modification of
\[
Z_s(t):=\int_s^tS(\kappa(t-r))B(u(r))\dd\mathbb W_r.
\]

For every nonzero Fourier mode $e_\ell$, test the variational equation by $e_\ell$, multiply by $e^{\kappa\lambda_\ell r}$ and integrate over $[s,t]$ to obtain
\begin{align*}
\langle u(t),e_\ell\rangle
={}&e^{-\kappa\lambda_\ell(t-s)}\langle u(s),e_\ell\rangle
+\int_s^t e^{-\kappa\lambda_\ell(t-r)}\langle\Dpx u(r),e_\ell\rangle\dd r\\
&+\sum_{k,i}\int_s^t e^{-\kappa\lambda_\ell(t-r)}
\langle\xi_{k,i}\cdot\nabla u(r),e_\ell\rangle\dd B_r^{k,i}.
\end{align*}
Partial Fourier sums of a smooth test converge in $V$. Drift integrability,~\eqref{eq:mild-HS} and density therefore give~\eqref{eq:mild} in $V^*$. A countable intersection over Fourier modes and dense times, followed by $V^*$ continuity, makes the identity valid for every $t\in[s,T]$ on one full-probability set. The deterministic convolution equals
\[
u(t)-S(\kappa(t-s))u(s)-Z_s(t),
\]
which is an $H$-continuous process. This proves the claimed $H$ representation. A further countable intersection handles any prescribed countable family of starting times.
\end{proof}

\section{Averaged dissipation enhancement}\label{sec:average}

Throughout this section, set $m=p^+$ and use $\eta<1$ from~\eqref{eq:eta-def-intro}.
Fix $R>0$, $\mu>0$, $t_0>0$, and define
\[
t_n:=nt_0,
\qquad n\in\mathbb N_0.
\]

For later use, define
\begin{equation}\label{eq:K2-def}
K_2(\kappa,t_0,R):=C_0\kappa^{-\eta}
\left(t_0^{1/p^- -\eta}R^{1-2/p^-}
+t_0^{1/p^+ -\eta}R^{1-2/p^+}\right),
\end{equation}
where $C_0:=C/(1-\eta)$ and $C>0$ depends on $d$, $p^-$, $p^+$ and the variable-exponent embedding constants. The factor $(1-\eta)^{-1}$ comes from integrating the kernel in~\eqref{eq:v2-time}. For fixed $t_0$ and $R$, one has $K_2(\kappa,t_0,R)\to0$ as $\kappa\to\infty$.

Let $\beta>d/2+1$ and choose $\alpha>1$ with $1/\alpha+1/\beta=1$.
Define
\begin{equation*}
\mathscr L
:=\left(\sum_{\ell\in\mathbb Z_0^d}\lambda_\ell^{1-\beta}\right)^{1/\beta}<\infty
\end{equation*}
and
\begin{equation}\label{eq:Gamma-def}
\Gamma(\mu,R)
:=2\hat{C}_e^2\left(
\mu^{2/p^+}
+\mu^{2/p^-}R^{2(p^+/p^- -1)}
\right).
\end{equation}
Finally, set
\begin{align}
C_1(\kappa,t_0,R)
&:=\left[
6\left(\frac1{(4\pi^2\kappa t_0)^2}
+K_2(\kappa,t_0,R)^2\right)
\right]^{(m-2)/2},\label{eq:C1-def}\\
C_2(\theta,\mu,R)
&:=\left[
6d\,\Gamma(\mu,R)^{1/\alpha}
\mathscr L\norm{\theta}_{\ell^\infty}^{2/\beta}
\right]^{(m-2)/2}.\label{eq:C2-def}
\end{align}

\begin{proposition}\label{prop:one-step}
Let $\norm{u_0}_{L^2}\le R$.  Suppose that $\kappa$ and $\theta$ are chosen so that
\begin{equation}\label{eq:smallness-one-step}
\max\{C_1(\kappa,t_0,R),C_2(\theta,\mu,R)\}
\le\frac1{1+\mu(m-2)t_0R^{m-2}}.
\end{equation}
Then, for every $n\in\mathbb N_0$,
\begin{equation}\label{eq:one-step}
\mathbb{E}\norm{u(t_{n+1})}_{L^2}^2
\le
\frac{\mathbb{E}\norm{u(t_n)}_{L^2}^2}
{\left(1+\mu(m-2)t_0
\bigl(\mathbb{E}\norm{u(t_n)}_{L^2}^2\bigr)^{(m-2)/2}
\right)^{2/(m-2)}}.
\end{equation}
\end{proposition}

For a fixed $n$, denote
\begin{equation*}
X_n:=\mathbb{E}\norm{u(t_n)}_{L^2}^2,
\qquad
\mathcal D_n:=\mathbb{E}\int_{t_n}^{t_{n+1}}
\rho_{p(\cdot)}(\nabla u(s))\dd s.
\end{equation*}
We split the proof into the following two cases:
\begin{align}
\text{Case 1:}\quad &\mathcal D_n\ge\mu t_0X_n^{m/2},\label{eq:case1}\\
\text{Case 2:}\quad &\mathcal D_n<\mu t_0X_n^{m/2}.\label{eq:case2}
\end{align}

\subsection{Convolution estimates}

For $r\in[t_n,t_{n+1}]$, define
\begin{align*}
v_1(r)&:=e^{\kappa(r-t_n)\Delta}u(t_n),\\
v_2(r)&:=\int_{t_n}^r e^{\kappa(r-\tau)\Delta}\Dpx u(\tau)\dd\tau,\\
v_3(r)&:=\sum_{k,i}\int_{t_n}^r e^{\kappa(r-\tau)\Delta}
\bigl(\xi_{k,i}\cdot\nabla u(\tau)\bigr)\dd B_\tau^{k,i}.
\end{align*}
By Theorem~\ref{thm:mild}, upon choosing the common full-probability set associated with the grid times,
\[
u(r)=v_1(r)+v_2(r)+v_3(r)
\]
in $H$ for every $r\in[t_n,t_{n+1}]$.
Since $t\mapsto\norm{u(t)}_{L^2}$ is nonincreasing, it follows that
\begin{equation}\label{eq:average-triangle}
\begin{aligned}
\norm{u(t_{n+1})}_{L^2}
&\le\frac1{t_0}\int_{t_n}^{t_{n+1}}\norm{u(r)}_{L^2}\dd r\le\sum_{j=1}^3\frac1{t_0}\int_{t_n}^{t_{n+1}}
\norm{v_j(r)}_{L^2}\dd r.
\end{aligned}
\end{equation}

\begin{lemma}\label{lem:v2}
Almost surely, the following estimate holds
\begin{equation}\label{eq:v2-est}
\frac1{t_0}\int_{t_n}^{t_{n+1}}\norm{v_2(r)}_{L^2}\dd r
\le K_2(\kappa,t_0,R)\norm{u(t_n)}_{L^2}.
\end{equation}
\end{lemma}

\begin{proof}
Set $D(\tau)=\rho_{p(\cdot)}(\nabla u(\tau))$. For almost every $r\in(t_n,t_{n+1})$, $v_2(r)$ has zero mean. By duality on $H$, together with Lemma~\ref{lem:heat-variable} and Lemma~\ref{lem:flux-norm}, we obtain
\begin{align*}
\|v_2(r)\|_{L^2}
&\le\int_{t_n}^r\sup_{\substack{\phi\in H\\\|\phi\|_{L^2}=1}}
|\langle\Dpx u(\tau),S(\kappa(r-\tau))\phi\rangle|\dd\tau\\
&\le C\kappa^{-\eta}\int_{t_n}^r(r-\tau)^{-\eta}
\bigl(D(\tau)^{a_-}+D(\tau)^{a_+}\bigr)\dd\tau.
\end{align*}
By Tonelli's theorem and the condition $\eta<1$, we have
\begin{align}
\frac1{t_0}\int_{t_n}^{t_{n+1}}\|v_2(r)\|_{L^2}\dd r
&\le\frac{C\kappa^{-\eta}}{(1-\eta)t_0}
\int_{t_n}^{t_{n+1}}(t_{n+1}-\tau)^{1-\eta}
\bigl(D^{a_-}+D^{a_+}\bigr)\dd\tau\notag\\
&\le C_0\kappa^{-\eta}t_0^{-\eta}
\int_{t_n}^{t_{n+1}}\bigl(D^{a_-}+D^{a_+}\bigr)\dd\tau.\label{eq:v2-time}
\end{align}
For each $r\in\{p^-,p^+\}$, H\"older's inequality together with the pathwise energy identity yields
\[
\int_{t_n}^{t_{n+1}}D^{1-1/r}\dd\tau
\le t_0^{1/r}\left(\int_{t_n}^{t_{n+1}}D\dd\tau\right)^{1-1/r}
\le t_0^{1/r}\left(\frac{\|u(t_n)\|_{L^2}^2}{2}\right)^{1-1/r}.
\]
Consequently,
\begin{align*}
\frac1{t_0}\int_{t_n}^{t_{n+1}}\|v_2(r)\|_{L^2}\dd r
\le C_0\kappa^{-\eta}\bigl(&t_0^{1/p^- -\eta}\|u(t_n)\|_{L^2}^{2-2/p^-}\\
&+t_0^{1/p^+ -\eta}\|u(t_n)\|_{L^2}^{2-2/p^+}\bigr).
\end{align*}
Since $p^\pm>2$ and $\|u(t_n)\|_{L^2}\le R$, we have
$\|u(t_n)\|_{L^2}^{2-2/p^\pm}\le R^{1-2/p^\pm}\|u(t_n)\|_{L^2}$.
Substituting the preceding estimates yields~\eqref{eq:v2-est} with $K_2$ as in~\eqref{eq:K2-def}.
\end{proof}

\begin{lemma}\label{lem:v3}
Assume that~\eqref{eq:case2} holds.  Then
\begin{equation}\label{eq:v3-est}
\mathbb{E}\left(
\frac1{t_0}\int_{t_n}^{t_{n+1}}\norm{v_3(r)}_{L^2}\dd r
\right)^2
\le d\,\Gamma(\mu,R)^{1/\alpha}
\mathscr L\norm{\theta}_{\ell^\infty}^{2/\beta}X_n.
\end{equation}
\end{lemma}

\begin{proof}
An application of Jensen's inequality yields
\begin{equation}\label{eq:v3-jensen}
\mathbb{E}\left(
\frac1{t_0}\int_{t_n}^{t_{n+1}}\norm{v_3(r)}_{L^2}\dd r
\right)^2
\le\frac1{t_0}\int_{t_n}^{t_{n+1}}
\mathbb{E}\norm{v_3(r)}_{L^2}^2\dd r.
\end{equation}
By It\^o's isometry and Parseval's identity, we have
\begin{align}
\mathbb{E}\norm{v_3(r)}_{L^2}^2
&=\mathbb{E}\sum_{k,i}\int_{t_n}^{r}
\norm{e^{\kappa(r-\tau)\Delta}
(\xi_{k,i}\cdot\nabla u(\tau))}_{L^2}^2\dd\tau\notag\\
&=\mathbb{E}\sum_{k,i}\sum_{\ell\in\mathbb Z_0^d}
\int_{t_n}^{r}e^{-2\kappa\lambda_\ell(r-\tau)}
\abs{\pair{\xi_{k,i}\cdot\nabla u(\tau)}{e_\ell}}^2\dd\tau.
\label{eq:v3-parseval}
\end{align}
Define
\begin{align*}
I_1&:=\mathbb{E}\sum_{k,i,\ell}
\int_{t_n}^{t_{n+1}}\int_{t_n}^{r}
 e^{-2\kappa\lambda_\ell(r-\tau)}\lambda_\ell
\abs{\pair{\xi_{k,i}\cdot\nabla u(\tau)}{e_\ell}}^2
\dd\tau\dd r,
\\
I_2&:=\mathbb{E}\sum_{k,i,\ell}
\int_{t_n}^{t_{n+1}}\int_{t_n}^{r}
 e^{-2\kappa\lambda_\ell(r-\tau)}\lambda_\ell^{1-\beta}
\abs{\pair{\xi_{k,i}\cdot\nabla u(\tau)}{e_\ell}}^2
\dd\tau\dd r.
\end{align*}

We apply H\"older's inequality with exponents $\alpha,\beta$ on the product of probability, time and counting measures, first with finite Fourier sums. The weights satisfy
\[
\lambda_\ell^{1/\alpha}\lambda_\ell^{(1-\beta)/\beta}=1.
\]
Thus, by monotone convergence of the truncated sums, the integral over $r$ of the right-hand side of~\eqref{eq:v3-parseval} is at most $I_1^{1/\alpha}I_2^{1/\beta}$. With~\eqref{eq:v3-jensen}, we therefore obtain
\begin{equation}\label{eq:v3-I1I2}
\mathbb E\left(\frac1{t_0}\int_{t_n}^{t_{n+1}}\|v_3(r)\|_{L^2}\dd r\right)^2
\le\frac1{t_0}I_1^{1/\alpha}I_2^{1/\beta}.
\end{equation}

\noindent\textbf{Step 1: Estimate of $I_1$.}  By the real formulation~\eqref{eq:xi-def} and the same cosine--sine pairing used in~\eqref{eq:covariance-xi}, we obtain
\begin{align}
I_1
&=\frac{2\kappa d}{d-1}
\mathbb{E}\sum_{k\in\mathbb Z_0^d}\theta_k^2
\sum_{i=1}^{d-1}\sum_{\ell\in\mathbb Z_0^d}
\int_{t_n}^{t_{n+1}}
\abs{\pair{a_{k,i}e_k\cdot\nabla u(\tau)}{e_\ell}}^2
\int_{\tau}^{t_{n+1}}
e^{-2\kappa\lambda_\ell(r-\tau)}\lambda_\ell\dd r\dd\tau.
\label{eq:I1-rewrite}
\end{align}
Thanks to
\[
\int_{\tau}^{t_{n+1}}
e^{-2\kappa\lambda_\ell(r-\tau)}\lambda_\ell\dd r
=\frac{1-e^{-2\kappa\lambda_\ell(t_{n+1}-\tau)}}{2\kappa}
\le\frac1{2\kappa},
\]
we obtain
\begin{align}
I_1
&\le\frac d{d-1}
\mathbb{E}\sum_k\theta_k^2\sum_{i=1}^{d-1}
\int_{t_n}^{t_{n+1}}
\sum_\ell\abs{\pair{a_{k,i}e_k\cdot\nabla u(\tau)}{e_\ell}}^2
\dd\tau\notag\\
&=\frac d{d-1}
\mathbb{E}\sum_k\theta_k^2\sum_{i=1}^{d-1}
\int_{t_n}^{t_{n+1}}
\norm{a_{k,i}e_k\cdot\nabla u(\tau)}_{L^2}^2\dd\tau\notag\\
&\le d\mathbb{E}\int_{t_n}^{t_{n+1}}\norm{\nabla u(\tau)}_{L^2}^2\dd\tau.
\label{eq:I1-grad2}
\end{align}
Using~\eqref{eq:embedding-L2} and~\eqref{eq:modular-norm-2}, together with the elementary inequality $(a+b)^2\le2(a^2+b^2)$, we get
\begin{equation}\label{eq:grad2-D}
\|\nabla u(\tau)\|_{L^2}^2
\le2\hat{C}_e^2\bigl(D(\tau)^{2/p^-}+D(\tau)^{2/p^+}\bigr).
\end{equation}

For each $r\in\{p^-,p^+\}$, Jensen's inequality on $\Omega\times(t_n,t_{n+1})$ together with the small-dissipation hypothesis yields
\[
\mathbb E\int_{t_n}^{t_{n+1}}D(\tau)^{2/r}\dd\tau
\le t_0^{1-2/r}\left(\mathbb E\int_{t_n}^{t_{n+1}}D(\tau)\dd\tau\right)^{2/r}
\le t_0\mu^{2/r}X_n^{p^+/r}.
\]
The term with $r=p^+$ is linear in $X_n$, whereas
$X_n^{p^+/p^-}\le R^{2(p^+/p^- -1)}X_n$ since $X_n\le R^2$.
Combining~\eqref{eq:I1-grad2},~\eqref{eq:grad2-D} and~\eqref{eq:Gamma-def}, it holds
\begin{equation}\label{eq:I1-final}
I_1\le2d\hat{C}_e^2t_0\left(\mu^{2/p^+}
+\mu^{2/p^-}R^{2(p^+/p^- -1)}\right)X_n
=d\Gamma(\mu,R)t_0X_n.
\end{equation}

\noindent\textbf{Step 2: Estimate of $I_2$.}  Using the convention for the complex inner product fixed above and the condition $a_{k,i}\cdot k=0$, a direct integration by parts yields
\begin{align*}
\pair{a_{k,i}e_k\cdot\nabla u}{e_\ell}_{L^2}
&=\int_{\mathbb T^d}(a_{k,i}\cdot\nabla u)e_k\overline{e_\ell}\dd x\\
&=-\int_{\mathbb T^d}u\,a_{k,i}\cdot\nabla(e_k\overline{e_\ell})\dd x=2\pi i(a_{k,i}\cdot\ell)\pair{ue_k}{e_\ell}_{L^2}.
\end{align*}
Thus, we get
\begin{equation}\label{eq:integration-by-parts-fourier}
\abs{\pair{a_{k,i}e_k\cdot\nabla u}{e_\ell}}^2
=(2\pi)^2(a_{k,i}\cdot\ell)^2
\abs{\pair{ue_k}{e_\ell}}^2
\le\lambda_\ell\abs{\pair{ue_k}{e_\ell}}^2.
\end{equation}
Proceeding as in~\eqref{eq:I1-rewrite} and using~\eqref{eq:integration-by-parts-fourier}, we obtain
\begin{align*}
I_2
&\le\frac{2\kappa d}{d-1}
\mathbb{E}\sum_{k}\theta_k^2\sum_{i=1}^{d-1}\sum_\ell
\int_{t_n}^{t_{n+1}}
\abs{\pair{u(\tau)e_k}{e_\ell}}^2
\int_{\tau}^{t_{n+1}}
e^{-2\kappa\lambda_\ell(r-\tau)}
\lambda_\ell^{2-\beta}\dd r\dd\tau.
\end{align*}
Since
\[
\int_{\tau}^{t_{n+1}}
e^{-2\kappa\lambda_\ell(r-\tau)}
\lambda_\ell^{2-\beta}\dd r
\le\frac1{2\kappa}\lambda_\ell^{1-\beta},
\]
we derive
\begin{align*}
I_2
&\le d\mathbb{E}\sum_k\theta_k^2\sum_\ell\lambda_\ell^{1-\beta}
\int_{t_n}^{t_{n+1}}
\abs{\pair{u(\tau)e_k}{e_\ell}}^2\dd\tau\\
&\le d\norm{\theta}_{\ell^\infty}^2
\mathbb{E}\sum_\ell\lambda_\ell^{1-\beta}
\int_{t_n}^{t_{n+1}}
\sum_k\abs{\pair{u(\tau)e_k}{e_\ell}}^2\dd\tau.
\end{align*}
For fixed $\ell$, the standard complex $L^2$ pairing yields
\[
\pair{u(\tau)e_k}{e_\ell}
=\pair{u(\tau)}{e_{\ell-k}}.
\]
Since the map $k\mapsto\ell-k$ is injective on $\mathbb Z_0^d$, Parseval's identity yields
\begin{equation*}
\sum_{k\in\mathbb Z_0^d}\abs{\pair{u(\tau)e_k}{e_\ell}}^2
=\sum_{k\in\mathbb Z_0^d}\abs{\pair{u(\tau)}{e_{\ell-k}}}^2
\le\sum_{j\in\mathbb Z^d}\abs{\pair{u(\tau)}{e_j}}^2
=\norm{u(\tau)}_{L^2}^2.
\end{equation*}
Using the pathwise monotonicity of the $L^2$ norm, we have
\begin{align}
I_2
&\le d\norm{\theta}_{\ell^\infty}^2
\left(\sum_\ell\lambda_\ell^{1-\beta}\right)
\mathbb{E}\int_{t_n}^{t_{n+1}}\norm{u(\tau)}_{L^2}^2\dd\tau\le d\norm{\theta}_{\ell^\infty}^2\mathscr L^\beta t_0X_n.
\label{eq:I2-final}
\end{align}
Finally, substituting~\eqref{eq:I1-final} and~\eqref{eq:I2-final} into~\eqref{eq:v3-I1I2} gives
\begin{align*}
\frac1{t_0}I_1^{1/\alpha}I_2^{1/\beta}
&\le\frac1{t_0}
\bigl(d\Gamma(\mu,R)t_0X_n\bigr)^{1/\alpha}
\bigl(d\norm{\theta}_{\ell^\infty}^2
\mathscr L^\beta t_0X_n\bigr)^{1/\beta}\\
&=d\Gamma(\mu,R)^{1/\alpha}
\mathscr L\norm{\theta}_{\ell^\infty}^{2/\beta}X_n.
\end{align*}
This proves~\eqref{eq:v3-est}.
\end{proof}

\begin{proof}[Proof of Proposition~\ref{prop:one-step}]
If~\eqref{eq:case1} holds, the energy identity and~\eqref{eq:tangent} give
\begin{align}
X_{n+1}
&\le X_n\left(1-2\mu t_0X_n^{(m-2)/2}\right)\notag\\
&\le\frac{X_n}
{\left(1+\mu(m-2)t_0X_n^{(m-2)/2}\right)^{2/(m-2)}}.
\label{eq:case1-pre}
\end{align}
This case requires no smallness condition on the noise.

Suppose now that~\eqref{eq:case2} holds. The spectral gap gives
\begin{equation}\label{eq:v1-est}
\frac1{t_0}\int_{t_n}^{t_{n+1}}\|v_1(r)\|_{L^2}\dd r
\le\frac{\|u(t_n)\|_{L^2}}{4\pi^2\kappa t_0}.
\end{equation}
By~\eqref{eq:average-triangle},~\eqref{eq:v1-est} and Lemmas~\ref{lem:v2}--\ref{lem:v3}, we obtain
\begin{align}
X_{n+1}
&\le3\left[
\frac1{(4\pi^2\kappa t_0)^2}
+K_2(\kappa,t_0,R)^2
+d\Gamma(\mu,R)^{1/\alpha}
\mathscr L\norm{\theta}_{\ell^\infty}^{2/\beta}
\right]X_n\notag\\
&=\frac12\left[
C_1(\kappa,t_0,R)^{2/(m-2)}
+C_2(\theta,\mu,R)^{2/(m-2)}
\right]X_n\notag\\
&\le\max\{C_1(\kappa,t_0,R),C_2(\theta,\mu,R)\}^{2/(m-2)}X_n.
\label{eq:raw-contraction}
\end{align}
This estimate holds before imposing~\eqref{eq:smallness-one-step}. Under that condition, $X_n\le R^2$ yields
\[
X_{n+1}\le
\frac{X_n}{(1+\mu(m-2)t_0R^{m-2})^{2/(m-2)}}
\le
\frac{X_n}{(1+\mu(m-2)t_0X_n^{(m-2)/2})^{2/(m-2)}}.
\]
Thus~\eqref{eq:one-step} holds in both cases.
\end{proof}

\subsection{Choice of the noise}

\begin{proposition}\label{prop:choose-noise}
Suppose that Assumption~\ref{ass:p} holds and fix $\beta>d/2+1$. For every $t_0,R,\mu>0$, there exist $\kappa>0$ and a finitely supported annular profile $\theta$ satisfying~\eqref{eq:theta-assumptions} and~\eqref{eq:smallness-one-step}. The choice depends only on these parameters and the structural constants.
\end{proposition}

\begin{proof}
For fixed $t_0,R,\mu$,~\eqref{eq:K2-def}--\eqref{eq:C2-def} give
\[
C_1(\kappa,t_0,R)\longrightarrow0\quad(\kappa\to\infty),
\qquad
C_2(\theta,\mu,R)\longrightarrow0\quad(\|\theta\|_{\ell^\infty}\to0).
\]
Choose $\kappa$ and then an annular profile from Corollary~\ref{cor:finite-annular} so that each quantity is at most $(1+\mu(m-2)t_0R^{m-2})^{-1}$.
\end{proof}

\subsection{Proof of averaged enhanced dissipation}

\begin{proof}[Proof of Theorem~\ref{thm:average-main}]
Choose $t_0\in(0,1)$ and $\widehat\mu$ with $\widehat\mu(1-t_0)>\mu$, and apply Proposition~\ref{prop:choose-noise} with $\widehat\mu$. Set $X_j=\mathbb E\|u(t_j)\|_{L^2}^2$. Proposition~\ref{prop:one-step} and Lemma~\ref{lem:iteration} then yield
\begin{equation}\label{eq:iteration-final-average}
X_n\le\frac{\|u_0\|_{L^2}^2}
{\bigl(1+\widehat\mu(m-2)nt_0\|u_0\|_{L^2}^{m-2}\bigr)^{2/(m-2)}}.
\end{equation}
For $t\ge1$, choose $n$ such that $t_n\le t<t_n+t_0$. Then
\[
t_n\ge t-t_0\ge(1-t_0)t,\qquad \widehat\mu t_n\ge\mu t.
\]
By monotonicity, we have $\mathbb E\|u(t)\|_{L^2}^2\le X_n$. Substituting into~\eqref{eq:iteration-final-average} and taking square roots proves~\eqref{eq:average-main}.
\end{proof}

\subsection{Strict comparison with the unperturbed evolution}\label{subsec:strict}

\begin{proof}[Proof of Theorem~\ref{thm:strict-enhancement}]
Fix $T>0$ and set $\Lambda_*=p^+M/(p^-r_*^2)$. By~\eqref{eq:def-Lambda0} and Theorem~\ref{thm:lower-bound}, we have
\[
\|u^{\mathrm{det}}(T)\|_{L^2}
\ge\frac{\|u_0\|_{L^2}}{(1+(p^--2)\Lambda_*T)^{1/(p^--2)}}
\qquad(u_0\in\mathcal S(r_*,R,M)).
\]
Choose $\mu>0$ so that
\[
(1+(p^+-2)\mu T r_*^{p^+-2})^{1/(p^+-2)}
\ge\lambda^{-1}(1+(p^--2)\Lambda_*T)^{1/(p^--2)}.
\]
Apply Proposition~\ref{prop:choose-noise} with $t_0=T$. Proposition~\ref{prop:one-step} at $n=0$ then gives
\[
\|u(T)\|_{L^2(\Omega;L^2)}
\le\frac{\|u_0\|_{L^2}}
{(1+(p^+-2)\mu T\|u_0\|_{L^2}^{p^+-2})^{1/(p^+-2)}}.
\]
Since $\|u_0\|_{L^2}\ge r_*$, we obtain
\begin{align*}
\|u(T)\|_{L^2(\Omega;L^2)}
&\le\frac{\|u_0\|_{L^2}}{(1+(p^+-2)\mu T r_*^{p^+-2})^{1/(p^+-2)}}\\
&\le\lambda\frac{\|u_0\|_{L^2}}{(1+(p^--2)\Lambda_*T)^{1/(p^--2)}}
\le\lambda\|u^{\mathrm{det}}(T)\|_{L^2}.
\end{align*}
The noise depends only on the prescribed $T,\lambda,r_*,R,M$ and the structural data.
\end{proof}

\begin{remark}\label{rem:why-lower-bound}
The lower bound in Theorem~\ref{thm:lower-bound} is needed to compare the two evolutions. The estimate $\Lambda_0(u_0)\le p^+M/(p^-r_*^2)$ uses both the gradient modular bound and the lower cutoff on $\|u_0\|_{L^2}$. Since $\Lambda_0$ is unbounded on $L^2$-balls, this argument does not yield the same uniform comparison from an $L^2$ bound alone.
\end{remark}

\section{Exponential decay and pathwise dissipation enhancement}\label{sec:pathwise}

Fix $\kappa>0$ and let $\theta^{(n)}=(\theta_{k,n})_{k\in\mathbb Z_0^d}$ satisfy~\eqref{eq:theta-assumptions}. We use the deterministic stepwise profiles

\begin{equation}\label{eq:theta-step}
\theta_k(t):=\theta_{k,n},
\qquad t\in[n,n+1).
\end{equation}
Define $\xi_{k,i}(x,t)$ by~\eqref{eq:xi-def} with $\theta_k$ replaced by $\theta_k(t)$. On every interval $[n,n+1)$, the correction remains $\kappa\Delta$, since the profile is normalized and isotropic. Thus the equation reads
\begin{equation}\label{eq:time-dependent-spde}
\dd u=\Dpx u\dd t+\kappa\Delta u\dd t
+\sum_{k,i}\xi_{k,i}(x,t)\cdot\nabla u\dd B_t^{k,i}.
\end{equation}

\begin{proposition}\label{prop:stepwise-wp}
Let $p$ be periodic and log-H\"older continuous, with $2<p^-\le p^+<\infty$. Let $u_0\in L^2(\Omega,\mathcal F_0,\mathbb P;H)$, and let the deterministic profiles $\{\theta^{(n)}\}_{n\ge0}$ satisfy~\eqref{eq:theta-assumptions}. With $\theta(t)$ defined by~\eqref{eq:theta-step}, equation~\eqref{eq:time-dependent-spde} has a unique global variational solution
$u\in L^2(\Omega;C([0,T];H))$
for every $T>0$. Almost surely, for all $0\le s\le t$,
\begin{equation}\label{eq:stepwise-energy}
\norm{u(t)}_H^2
+2\int_s^t\rho_{p(\cdot)}(\nabla u(r))\dd r
=\norm{u(s)}_H^2.
\end{equation}
If, in addition, $p^+<p_{\mathrm c}$, the shifted mild formula of Theorem~\ref{thm:mild} holds on each $[n,n+1]$ with profile $\theta^{(n)}$.
\end{proposition}

\begin{proof}
Given $u(n)\in L^2(\Omega,\mathcal F_n;H)$, apply Theorem~\ref{thm:wellposed} with filtration $\mathcal F_{n+r}$, Wiener increments $B_{n+r}^{k,i}-B_n^{k,i}$ and profile $\theta^{(n)}$. Concatenation and uniqueness give the global solution, and summing the interval energy identities gives~\eqref{eq:stepwise-energy}. Under the additional upper bound on $p^+$, Theorem~\ref{thm:mild} gives the shifted representation, including the right endpoint by $H$-continuity. A countable intersection makes the assertions simultaneous on all intervals.
\end{proof}

\begin{proposition}\label{prop:restart-one-step}
Suppose that Assumption~\ref{ass:p} holds. Fix $s\ge0$, $t_0,R,\mu>0$, and let $v$ be an $\mathcal F_s$-measurable $H$-valued random variable with $\|v\|_H\le R$ almost surely. On $[s,s+t_0]$, solve~\eqref{eq:ito-main} with $u(s)=v$, a deterministic profile satisfying~\eqref{eq:theta-assumptions}, and the shifted Wiener increments. Set
\[
X_s:=\mathbb E\|v\|_H^2,\qquad
\mathcal D_s:=\mathbb E\int_s^{s+t_0}\rho_{p(\cdot)}(\nabla u(r))\dd r.
\]
The large-dissipation estimate~\eqref{eq:case1-pre} and the small-dissipation estimate~\eqref{eq:raw-contraction} hold with $X_n,\mathcal D_n$ replaced by $X_s,\mathcal D_s$. In particular, if~\eqref{eq:smallness-one-step} holds, then, with $m=p^+$,
\begin{equation}\label{eq:restart-one-step}
\mathbb E\|u(s+t_0)\|_H^2
\le\frac{X_s}
{\left(1+\mu(m-2)t_0X_s^{(m-2)/2}\right)^{2/(m-2)}}.
\end{equation}
\end{proposition}

\begin{proof}
The shifted energy identity gives $\|u(s+r)\|_H\le\|v\|_H\le R$. The proofs of Lemmas~\ref{lem:v2}--\ref{lem:v3} and Proposition~\ref{prop:one-step} use only this bound, predictability and It\^o's isometry, so they apply with the same constants to the random datum $v$.
\end{proof}

Following the intervalwise acceleration in~\cite[Proposition~5.1]{DongLuoTang2023}, fix $\gamma>2$ and set $\mu_n:=\gamma^n\mu_0$.

\begin{proposition}\label{prop:path-discrete}
Suppose that Assumption~\ref{ass:p} holds. Set $m=p^+$ and fix $R>0$, $\mu_0>0$ and $\gamma>2$.  There exist a fixed $\kappa>0$ and deterministic finitely supported profiles $\{\theta^{(n)}\}_{n\ge0}$, each satisfying~\eqref{eq:theta-assumptions} and depending only on $R$, $\mu_0$, $\gamma$, $d$, the fixed parameter $\beta$, and the exponent data, such that for every zero-mean $u_0\in L^2(\mathbb T^d)$ with $\norm{u_0}_{L^2}\le R$,
\begin{equation}\label{eq:path-one-step}
\mathbb{E}\norm{u(n)}_{L^2}^2
\le
\frac{\mathbb{E}\norm{u(n-1)}_{L^2}^2}
{\left(1+\mu_0(m-2)\gamma^{n-1}
\bigl(\mathbb{E}\norm{u(n-1)}_{L^2}^2\bigr)^{(m-2)/2}
\right)^{2/(m-2)}}
\end{equation}
for every $n\ge1$.  Therefore,
\begin{equation}\label{eq:path-discrete-sum}
\mathbb{E}\norm{u(n)}_{L^2}^2
\le
\frac{\norm{u_0}_{L^2}^2}
{\left(1+\mu_0(m-2)\dfrac{\gamma^n-1}{\gamma-1}
\norm{u_0}_{L^2}^{m-2}\right)^{2/(m-2)}}.
\end{equation}
\end{proposition}

\begin{proof}
Choose $\kappa$ and then each annular profile $\theta^{(n)}$ so that
\begin{align}
C_1&(\kappa,1,R)\le c_*,\label{eq:C1-path-choice}\\
C_2&(\theta^{(n)},\mu_n,R)\le c_*,\label{eq:C2-path-choice}
\end{align}
where $c_*:=\min\left\{(1+\mu_0(m-2)R^{m-2})^{-1},(\gamma+1)^{-1}\right\}$.
Since $C_1(\kappa,1,R)$ is independent of $\mu_n$ and tends to zero as $\kappa\to\infty$, one fixed $\kappa$ satisfies~\eqref{eq:C1-path-choice}. For each $n$, Corollary~\ref{cor:finite-annular} then gives~\eqref{eq:C2-path-choice} by decreasing $\|\theta^{(n)}\|_{\ell^\infty}$.

 Setting $Y_j=\mathbb E\|u(j)\|_{L^2}^2$ and $h=(m-2)/2$, we see that the restarted one-step estimate establishes both assertions for $n=1$. Suppose the sum estimate holds at $N-1\ge1$. The case $Y_N=0$ is immediate; otherwise $Y_{N-1}>0$ and
\begin{equation}\label{eq:inverse-induction}
Y_{N-1}^{-h}\ge\|u_0\|_{L^2}^{-(m-2)}
+\mu_0(m-2)\frac{\gamma^{N-1}-1}{\gamma-1}.
\end{equation}
Restart on $[N-1,N]$ with threshold $\mu_{N-1}$. In the large-dissipation case,~\eqref{eq:case1-pre} and Proposition~\ref{prop:restart-one-step} give
\[
Y_N^{-h}-Y_{N-1}^{-h}\ge\mu_0(m-2)\gamma^{N-1}.
\]
In the small-dissipation case,~\eqref{eq:raw-contraction} gives
\begin{equation*}
Y_N\le c_*^{1/h}Y_{N-1}.
\end{equation*}
Consequently, by~\eqref{eq:inverse-induction},
\begin{align*}
Y_N^{-h}-Y_{N-1}^{-h}
&\ge(c_*^{-1}-1)Y_{N-1}^{-h}\\
&\ge\gamma\mu_0(m-2)\frac{\gamma^{N-1}-1}{\gamma-1}
\ge\mu_0(m-2)\gamma^{N-1},
\end{align*}
where the last inequality uses $N\ge2$. This proves~\eqref{eq:path-one-step}. Adding the inverse increment to~\eqref{eq:inverse-induction} gives~\eqref{eq:path-discrete-sum}.
\end{proof}

\begin{remark}\label{rem:frequency-growth}
	Keep $R,\mu_0,\gamma$ and the structural data fixed. For $\mu_n\ge1$,~\eqref{eq:Gamma-def} gives $\Gamma(\mu_n,R)\le C\mu_n^{2/p^-}$. Since $\beta/\alpha=\beta-1$, a sufficient condition for~\eqref{eq:C2-path-choice} is
	\[
	\|\theta^{(n)}\|_{\ell^\infty}
	\le c\,\mu_n^{-(\beta-1)/p^-},
	\]
	with $c>0$ independent of $n$. For the annular profiles of Corollary~\ref{cor:finite-annular}, one has $|\mathcal K_N|\asymp N^d$, so it suffices to choose integer radii
	\[
	N_n\ge C\max\{1,\mu_n^{2(\beta-1)/(dp^-)}\}.
	\]
	The finitely many indices with $\mu_n<1$ can be treated separately. With $N_n\to\infty$, these profiles satisfy
	$\sum_{k,i}\|\nabla\xi_{k,i}(\cdot,t)\|_{L^\infty}^2\asymp\kappa N_n^2$ on $[n,n+1)$. Thus the intensity and the pointwise covariance $\sum_{k,i}\xi_{k,i}(x,t)\otimes\xi_{k,i}(x,t)=2\kappa I_d$ remain constant, while the frequency range and the spatial derivatives grow. The displayed scale is sufficient for the proof; no optimality is claimed.
\end{remark}

\begin{proof}[Proof of Theorem~\ref{thm:exponential-main}]
Set $m=p^+$, let $r=\max\{2,q\}$ and choose $\gamma>2$ with $\lambda:=\log\gamma/(m-2)>cr/2$.
Apply Proposition~\ref{prop:path-discrete} with $\mu_0=1$ and this $\gamma$. Since $(\gamma^n-1)/(\gamma-1)\ge\gamma^{n-1}$,~\eqref{eq:path-discrete-sum} gives, for $n\ge1$ and $u_0\ne0$,
\begin{equation*}
\mathbb E\|u(n)\|_{L^2}^2
\le(m-2)^{-2/(m-2)}\gamma^{-2(n-1)/(m-2)}
=\mathscr A e^{-2\lambda n},
\quad \mathscr A=(m-2)^{-2/(m-2)}e^{2\lambda}.
\end{equation*}
For $u_0=0$, the solution is zero. If $n\le t<n+1$, monotonicity gives $\|u(t)\|_{L^2}\le\|u(n)\|_{L^2}$. Thus~\eqref{eq:exponential-mean-main} holds with
\[
\widehat C=e^{2\lambda}\max\{R^2,\mathscr A\},
\]
since $\lambda>c$ and the initial interval is bounded by $R^2$.

Define the measurable random variable
\begin{equation*}
K(u_0,\omega):=\max\left\{1,e^c\sup_{n\ge0}e^{cn}\|u(n,\omega)\|_{L^2}\right\}.
\end{equation*}
Using $\|u(n)\|_{L^2}\le R$ and $r\ge2$, we obtain
\begin{align*}
\mathbb E K^r
&\le1+e^{cr}\sum_{n\ge0}e^{crn}\mathbb E\|u(n)\|_{L^2}^r\\
&\le1+e^{cr}R^{r-2}
\left[R^2+\mathscr A\sum_{n\ge1}e^{-(2\lambda-cr)n}\right]<\infty.
\end{align*}
The bound is uniform on $\|u_0\|_{L^2}\le R$. In particular, $K$ is almost surely finite, and Jensen's inequality yields $\mathbb E K^q\le(\mathbb E K^r)^{q/r}$. Finally, for $n\le t<n+1$,
\[
e^{ct}\|u(t)\|_{L^2}\le e^ce^{cn}\|u(n)\|_{L^2}\le K,
\]
which proves~\eqref{eq:exponential-path-main} and~\eqref{eq:exponential-moment-main} on a common full-probability set for all $t\ge0$. 
\end{proof}

The exponential estimate also yields the following pathwise algebraic bound.
Throughout the pathwise argument, we use the convention
\begin{equation}\label{eq:zero-exponential-convention}
\exp\left(-\frac1{0^+}\right):=0.
\end{equation}
With this convention,~\eqref{eq:path-main} remains valid if a trajectory reaches the zero state.

\begin{corollary}[Pathwise enhancement]\label{cor:path-main}
Suppose that Assumption~\ref{ass:p} holds. Fix $R>0$ and set $m=p^+$. Given $\mu_0,q>0$, there exist a fixed intensity $\kappa>0$ and a deterministic stepwise profile $\theta(t)$, finitely supported on each interval and satisfying~\eqref{eq:theta-assumptions}, with the following property. For every $u_0\in H$ with $0<\|u_0\|_{L^2}\le R$, the solution of~\eqref{eq:time-dependent-spde} satisfies
\begin{equation}\label{eq:path-main}
\exp\left(-\frac1{\norm{u(t)}_{L^2}^{m-2}}\right)
\le C(u_0,\omega)e^{-\mu_0(m-2)t}
\exp\left(-\frac1{\norm{u_0}_{L^2}^{m-2}}\right)
\end{equation}
for all $t\ge0$, almost surely, where $C=C(u_0,\omega)>0$ and $\mathbb E[C^q]<\infty$. Consequently, whenever
\[
t>\frac{\log C(u_0,\omega)}{\mu_0(m-2)},
\]
one has
\begin{equation}\label{eq:path-decay-main}
\norm{u(t)}_{L^2}
\le
\frac{\norm{u_0}_{L^2}}
{\left(1+\bigl(\mu_0(m-2)t-\log C(u_0,\omega)\bigr)
\norm{u_0}_{L^2}^{m-2}\right)^{1/(m-2)}}.
\end{equation}
For every $r_*\in(0,R]$, the same noise satisfies
\begin{equation}\label{eq:path-uniform-annulus}
\sup_{\substack{u_0\in H\\r_*\le\|u_0\|_{L^2}\le R}}
\mathbb{E}\bigl[C(u_0,\cdot)^q\bigr]<\infty.
\end{equation}
\end{corollary}

\begin{proof}
Apply Theorem~\ref{thm:exponential-main} with $c=\mu_0$ and moment order $q(m-2)$. For the resulting prefactor $K$, set $r_0=\|u_0\|_{L^2}>0$ and
\[
C(u_0,\omega):=\exp\bigl(r_0^{-(m-2)}-1\bigr)K(u_0,\omega)^{m-2}.
\]
Since $ye^{-y}\le e^{-1}$ for $y>0$, we have
\[
\exp\bigl(-\|u(t)\|_{L^2}^{-(m-2)}\bigr)
\le e^{-1}\|u(t)\|_{L^2}^{m-2}
\le e^{-1}K^{m-2}e^{-\mu_0(m-2)t},
\]
with zero trajectories covered by~\eqref{eq:zero-exponential-convention}. This gives~\eqref{eq:path-main}, and taking logarithms yields~\eqref{eq:path-decay-main} under its stated time condition. Finally,
\[
\sup_{r_*\le\|u_0\|_{L^2}\le R}\mathbb E[C^q]
\le e^{q(r_*^{-(m-2)}-1)}
\sup_{\|u_0\|_{L^2}\le R}\mathbb E[K^{q(m-2)}]<\infty
\]
by~\eqref{eq:exponential-moment-main}, proving both moment assertions.
\end{proof}

\appendix
\section{The variable exponent It\^o energy formula}\label{sec:appendix-framework}

We adapt the time discretization in~\cite[Theorem~A.2]{GwiazdaWoznickiWroblewskaZimmermann2025} to the periodic variable-exponent setting; see also~\cite{BauzetValletWittboldZimmermann2013}. Throughout this appendix, $p$ is periodic and log-H\"older continuous, with $2<p^-\le p^+<\infty$; the upper restriction $p^+<p_{\mathrm c}$ is not needed.

The cited formula uses an Orlicz gradient--flux pairing on a bounded domain with zero boundary values. Here we work directly in $L^{\widetilde p}(Q_T)$ and $L^{\widetilde p'}(Q_T)$. Time independence of $p$ gives the strong step approximations in Lemma~\ref{lem:px-time-approx}, while log-H\"older continuity gives the periodic spatial approximation used in Lemma~\ref{lem:right-endpoint-increment}. For the gradient-dependent noise, we require predictability
and square integrability of $B(u)$. Lemma~\ref{lem:transport-divergence} places the correction flux in $L^{\widetilde p'}$. Proposition~\ref{prop:variable-Ito} assumes weak $H$ continuity and the $L^2$ supremum bound, both obtained in Subsection~\ref{subsec:galerkin-limit} before flux identification. Its proof treats the internal and initial-time remainders separately and establishes the identity simultaneously at all times.
Let $\{\xi_j\}_{j\ge1}$ be a family of smooth divergence-free vector fields satisfying
\begin{equation}\label{eq:general-xi-sum}
\mathbf k:=\sum_{j=1}^\infty\norm{\xi_j}_{L^\infty}^2<\infty.
\end{equation}
By~\eqref{eq:xi-linfty-sum}, the Fourier fields~\eqref{eq:xi-def} satisfy~\eqref{eq:general-xi-sum}.

For $T>0$, write
\[
Q_T:=\Omega\times(0,T)\times\mathbb T^d,
\qquad
\widetilde p(\omega,t,x):=p(x),
\]
with measure $\dd\mathbb P\dd t\dd x$. We use the reflexive, separable product spaces $L^{\widetilde p}(Q_T)$ and $L^{\widetilde p'}(Q_T)$; e.g.~\cite{DieningNageleRuzicka2012,BauzetValletWittboldZimmermann2013,ValletZimmermann2016}.

Let $\mathbb W$ be a cylindrical Wiener process on $\mathcal U$ with basis $\{\mathbf e_j\}$. We use the transport operators
\begin{align}
\mathcal C_\xi(u)&:=\frac12\sum_j\xi_j\cdot\nabla(\xi_j\cdot\nabla u),\label{eq:Cxi-general}\\
B(u)\mathbf e_j&:=\xi_j\cdot\nabla u,\label{eq:B-general}
\end{align}
where the correction is defined variationally by
\begin{equation}\label{eq:Cxi-pairing}
\pair{\mathcal C_\xi(u)}{v}_{V^*,V}
:=-\frac12\sum_j\int_{\mathbb T^d}
(\xi_j\cdot\nabla u)(\xi_j\cdot\nabla v)\dd x.
\end{equation}
Condition~\eqref{eq:general-xi-sum} makes this pairing well defined for $u,v\in V$.

\begin{lemma}\label{lem:px-time-approx}
Let $r(\cdot)$ denote either $p(\cdot)$ or $p'(\cdot)$, and set
$\widetilde r(\omega,t,x):=r(x)$. We write
\[
\rho_{\widetilde r}(f):=
\mathbb{E}\int_{\mathbb R}\int_{\mathbb T^d}|f(\omega,t,x)|^{r(x)}\dd x\dd t
\]
for the modular of the zero extension in time.  Extend a function
$f\in L^{\widetilde r}(Q_T)$ by zero outside $(0,T)$ in the time variable and define
$\tau_hf(\omega,t,x):=f(\omega,t+h,x)$.  Then
\begin{equation}\label{eq:px-time-translation}
\norm{\tau_hf-f}_{L^{\widetilde r}(\Omega\times\mathbb R\times\mathbb T^d)}
\longrightarrow0
\quad\text{as }h\to0.
\end{equation}
Moreover, suppose that
\[
v\in L^2(\Omega\times(0,T);H),
\qquad
\nabla v\in L^{\widetilde p}(Q_T;\mathbb R^d),
\]
and that $v$ has an $H$-valued representative at every time.  There exists a sequence of nested partitions
\[
I_\ell=\{0=t_0^\ell<t_1^\ell<\cdots<t_{N_\ell}^\ell=T\},
\qquad
I_\ell\subset I_{\ell+1},
\qquad
\abs{I_\ell}:=\max_i(t_i^\ell-t_{i-1}^\ell)\longrightarrow0,
\]
whose interior nodes can be chosen from any prescribed full-measure set of times. Define the left and right endpoint approximations by
\begin{align*}
\widetilde v^\ell(t):=\sum_{i=2}^{N_\ell}
\mathbf 1_{(t_{i-1}^\ell,t_i^\ell]}(t)v(t_{i-1}^\ell),
\quad 
\widehat v^\ell(t):=\sum_{i=1}^{N_\ell-1}
\mathbf 1_{(t_{i-1}^\ell,t_i^\ell]}(t)v(t_i^\ell).
\end{align*}
We extend both step functions by zero outside $(0,T)$. The first left cell and the last right cell are also zero, so no $V$ regularity is required at $0,T$. These approximations satisfy
\begin{align}
&\widetilde v^\ell,\widehat v^\ell\longrightarrow v
~\text{strongly in }L^2(\Omega\times(0,T);H),\label{eq:px-step-H}\\
&\nabla\widetilde v^\ell,\nabla\widehat v^\ell\longrightarrow\nabla v
~\text{strongly in }L^{\widetilde p}(Q_T;\mathbb R^d).
\label{eq:px-step-grad}
\end{align}
If $v$ is adapted, then $\widetilde v^\ell$ is predictable.
\end{lemma}

\begin{proof}
Time translation is an isometry because $\widetilde r$ is independent of $t$. Fubini's theorem and translation continuity in $L^1(\mathbb R)$ give modular convergence for bounded simple functions of finite support. Since $1<r^-\le r^+<\infty$, density and the modular--norm inequalities yield~\eqref{eq:px-time-translation}.

We next choose the partitions.  Extend $v$ and $\nabla v$ by zero in time and, after rescaling, work on $[0,1]$. For $h_n:=2^{-n}$ and a shift $s\in[0,1)$, consider the grid $a_n(s)+h_n\mathbb Z$, where $a_n(s):=s-h_n\lfloor s/h_n\rfloor$, with $0$ and $1$ added at the ends. These grids are nested in $n$ for each fixed $s$, since $a_{n+1}(s)\equiv a_n(s)\pmod{h_{n+1}}$. Let $L_{n,s}v$ and $R_{n,s}v$ denote the corresponding left and right step approximations, extended by zero outside $(0,1)$, with the first left and last right cells also set to zero. Sampling at the left endpoint of the cell containing $t$ returns $\tau_{-r}v(t)$ for some $r\in[0,h_n)$; because $s\mapsto a_n(s)$ sweeps $[0,h_n)$ with multiplicity $h_n^{-1}$, Fubini's theorem therefore yields
\begin{align*}
\int_0^1\norm{L_{n,s}v-v}_{L^2(\Omega\times\mathbb R;H)}^2\dd s
&\le\frac1{h_n}\int_0^{h_n}
\norm{\tau_{-r}v-v}_{L^2(\Omega\times\mathbb R;H)}^2\dd r\\
&\le\sup_{|h|\le h_n}
\norm{\tau_hv-v}_{L^2(\Omega\times\mathbb R;H)}^2,
\\
\int_0^1\rho_{\widetilde p}\!\left(
\nabla L_{n,s}v-\nabla v\right)\dd s
&\le\frac1{h_n}\int_0^{h_n}
\rho_{\widetilde p}(\tau_{-r}\nabla v-\nabla v)\dd r\\
&\le\sup_{|h|\le h_n}
\rho_{\widetilde p}(\tau_h\nabla v-\nabla v),
\end{align*}
The change of variables $r\mapsto h_n-r$ gives the same two bounds for $R_{n,s}$, with $\tau_{-r}$ replaced by $\tau_r$.  All four averaged errors are thus controlled by the translation errors, which tend to zero by~\eqref{eq:px-time-translation}.

Choose $n_\ell$ so that the sum of the four averaged errors is at most $2^{-2\ell}$, and let $E_\ell(s)$ denote their pointwise sum. Since $\int_0^1\sum_\ell E_\ell(s)\dd s<\infty$, Tonelli's theorem implies $E_\ell(s)\to0$ for almost every shift. Given a prescribed full-measure set $\mathcal R\subset(0,T)$, the shifts for which one of the countably many selected nodes lies outside $\mathcal R$ form a null set; hence a countable intersection yields a single shift for which all four errors vanish and every interior node lies in $\mathcal R$. After rescaling, we obtain the nested partitions and~\eqref{eq:px-step-H}--\eqref{eq:px-step-grad}. If $v$ is adapted, each $v(t_{i-1}^\ell)$ is $\mathcal F_{t_{i-1}^\ell}$-measurable, and the interval convention $(t_{i-1}^\ell,t_i^\ell]$ makes $\widetilde v^\ell$ predictable.
\end{proof}

\begin{lemma}\label{lem:transport-divergence}
Suppose that~\eqref{eq:general-xi-sum} holds and let $u$ satisfy $\nabla u\in L^{\widetilde p}(Q_T;\mathbb R^d)$.  Define
\begin{equation}\label{eq:transport-flux-K}
\mathfrak K(u):=\frac12\sum_{j=1}^{\infty}\xi_j \bigl(\xi_j\cdot\nabla u\bigr).
\end{equation}
Then the series converges absolutely almost everywhere,
$
\abs{\mathfrak K(u)}\le\frac12\mathbf k\abs{\nabla u}$, 
and
\begin{equation}\label{eq:K-pprime}
\mathfrak K(u)\in L^2(Q_T;\mathbb R^d)\hookrightarrow L^{\widetilde p'}(Q_T;\mathbb R^d).
\end{equation}
Moreover, in the variational sense, $\mathcal C_\xi(u)=\Div\mathfrak K(u)$ in $V^*$ for almost every $(\omega,t)$.
\end{lemma}

\begin{proof}
The pointwise estimate follows from
\[
\sum_j\abs{\xi_j (\xi_j\cdot\nabla u)}
\le\left(\sum_j\norm{\xi_j}_{L^\infty}^2\right)\abs{\nabla u}.
\]
Since $(p')^+<2<p^-$ and $Q_T$ has finite measure,~\eqref{eq:K-pprime} follows. Absolute summability gives, for $v\in V$,
\begin{align*}
\pair{\Div\mathfrak K(u)}{v}_{V^*,V}
&=-\int_{\mathbb T^d}\mathfrak K(u)\cdot\nabla v\dd x\\
&=-\frac12\sum_j\int_{\mathbb T^d}(\xi_j\cdot\nabla u)(\xi_j\cdot\nabla v)\dd x
=\pair{\mathcal C_\xi(u)}{v}_{V^*,V}.
\end{align*}
\end{proof}

\begin{lemma}\label{lem:right-endpoint-increment}
Let $A\in L^{\widetilde p'}(Q_T;\mathbb R^d)$, and let $\mathbb M$ be an $H$-valued continuous martingale.  Suppose that, in $V^*$,
\begin{equation}\label{eq:abstract-divergence-total}
u(t)=u_0+\int_0^t\Div A(r)\dd r+\mathbb M(t),
\end{equation}
that $u$ is weakly continuous in $H$, and that $\nabla u\in L^{\widetilde p}(Q_T;\mathbb R^d)$.  Let $\mathcal R\subset(0,T)$ be a full-measure set such that $u(t)\in V$ almost surely for every $t\in\mathcal R$.  Then, for $t\in\mathcal R$ and $0\le s<t$, almost surely,
\begin{align}
\norm{u(t)}_H^2-\norm{u(s)}_H^2
={}&-2\int_s^t\int_{\mathbb T^d}A(r)\cdot\nabla u(t)\dd x\dd r
+2\pair{u(s)}{\mathbb M(t)-\mathbb M(s)}_H\notag\\
&+\norm{\mathbb M(t)-\mathbb M(s)}_H^2
-\norm{u(t)-u(s)-\mathbb M(t)+\mathbb M(s)}_H^2.
\label{eq:discrete-increment-identity}
\end{align}
\end{lemma}

\begin{proof}
For almost every $\omega$, since $A(\omega,\cdot,\cdot)\in L^{p'(\cdot)}((0,T)\times\mathbb T^d)$, Jensen's inequality in time implies that $\int_s^t\Div A(r)\dd r\in V^*$. Fix $t\in\mathcal R$. 
By~\eqref{eq:abstract-divergence-total}, we have
\[
\int_s^t\Div A(r)\dd r
=u(t)-u(s)-\bigl(\mathbb M(t)-\mathbb M(s)\bigr)
\quad\text{in }V^*.
\]
The right-hand side belongs to $H$. Thus the drift increment,
initially defined in $V^*$, is identified with an element of $H$.
In particular, its $H$-norm and the Hilbert space polarization
identity used below are well defined.
Periodic Friedrichs mollification then gives
\[
J_\varepsilon u(t)\to u(t)\quad\hbox{in }H,
\quad
\nabla J_\varepsilon u(t)\to\nabla u(t)\quad\hbox{in }L^{p(\cdot)}.
\]
We test the increment of~\eqref{eq:abstract-divergence-total} by $J_\varepsilon u(t)$ and let $\varepsilon\downarrow0$. Variable-exponent H\"older justifies the drift limit. The resulting identity is the polarization identity for
$u(t)-u(s)=\int_s^t\Div A(r)\dd r+\mathbb M(t)-\mathbb M(s)$, which shows~\eqref{eq:discrete-increment-identity}.
\end{proof}

\begin{lemma}\label{lem:Hilbert-discrete-qv}
Let
\[
\mathbb M(t)=\int_0^t\Phi(s)\,\dd\mathbb W_s,
\qquad
\Phi\in L^2_{\mathcal P}(\Omega\times(0,T);\mathcal L_2(\mathcal U,H)),
\]
and let $I_\ell=\{0=t_0^\ell<\cdots<t_{N_\ell}^\ell=T\}$ be deterministic partitions with $|I_\ell|\to0$.  For every fixed $t$ which is a node of all sufficiently fine partitions,
\begin{equation}\label{eq:Hilbert-discrete-qv}
\sum_{t_{i+1}^\ell\le t}\|\mathbb M(t_{i+1}^\ell)-\mathbb M(t_i^\ell)\|_H^2
\longrightarrow
\int_0^t\|\Phi(s)\|_{\mathcal L_2(\mathcal U,H)}^2\,\dd s
\end{equation}
in probability.  The same conclusion holds after applying any fixed bounded orthogonal projection to $\mathbb M$ and $\Phi$.
\end{lemma}

\begin{proof}
Let $\Pi_N$ project onto the first $N$ vectors of an $H$-orthonormal basis and write $\Delta_i^\ell\mathbb M=\mathbb M(t_{i+1}^\ell)-\mathbb M(t_i^\ell)$. Finite-dimensional quadratic variation gives the assertion for $\Pi_N\mathbb M$; see~\cite{PrevotRockner2007,LiuRockner2015}. Orthogonality and It\^o's isometry give the uniform tail bound
\[
\mathbb E\sum_{t_{i+1}^\ell\le t}
\|(I-\Pi_N)\Delta_i^\ell\mathbb M\|_H^2
=\mathbb E\int_0^t\|(I-\Pi_N)\Phi(s)\|_{\mathcal L_2}^2\dd s
\longrightarrow0.
\]
The limiting integrals have the same tail bound. Letting $\ell\to\infty$ and then $N\to\infty$ proves~\eqref{eq:Hilbert-discrete-qv}. Replacing $\Phi$ by $P\Phi$ gives the assertion for an orthogonal projection $P$.
\end{proof}

\begin{lemma}\label{lem:internal-remainder}
Assume the hypotheses of Lemma~\ref{lem:right-endpoint-increment} and, in addition, that
$\mathbb M(t)=\int_0^tB(u(s))\dd\mathbb W_s$ with $B(u)\in L^2_{\mathcal P}(\Omega\times(0,T);\mathcal L_2(\mathcal U,H))$. Let $I_\ell$ be the nested partitions of Lemma~\ref{lem:px-time-approx}, with all interior nodes lying in $\mathcal R$. For a fixed interior node $t=t_{k_\ell}^\ell$ that belongs to all sufficiently fine partitions, set
\begin{equation}\label{eq:internal-remainder}
\mathscr R_\ell(t):=\sum_{i=1}^{k_\ell-1}\norm{d_i^\ell}_H^2,
\qquad
d_i^\ell:=u(t_{i+1}^\ell)-u(t_i^\ell)-\bigl(\mathbb M(t_{i+1}^\ell)-\mathbb M(t_i^\ell)\bigr).
\end{equation}
Then
\begin{equation}\label{eq:internal-remainder-zero}
\mathscr R_\ell(t)\longrightarrow0
\quad\text{in probability}.
\end{equation}
\end{lemma}

\begin{proof}
Let $\Pi_N$ be the orthogonal projection onto the span of the first $N$ elements of an $H$-orthonormal basis of smooth periodic zero-mean functions.  For $i\ge1$, both endpoints belong to $\mathcal R$ and
\[
d_i^\ell=\int_{t_i^\ell}^{t_{i+1}^\ell}\Div A(s)\dd s\quad\text{in }V^*.
\]
Writing $\Delta_i\mathbb M=\mathbb M(t_{i+1}^\ell)-\mathbb M(t_i^\ell)$, we obtain
\begin{align*}
\norm{d_i^\ell}_H^2
=-\int_{t_i^\ell}^{t_{i+1}^\ell}\int_{\mathbb T^d}
A(s)\cdot\nabla\bigl(u(t_{i+1}^\ell)-u(t_i^\ell)-\Pi_N\Delta_i\mathbb M\bigr)\dd x\dd s+\pair{d_i^\ell}{(\Pi_N-I)\Delta_i\mathbb M}_H.
\end{align*}
Summing in $i$ and using Cauchy--Schwarz yields
\begin{equation}\label{eq:remainder-projection-expanded}
\mathscr R_\ell(t)
\le \varepsilon_{\ell,N}(t)
+\bigl(\mathscr R_\ell(t)\bigr)^{1/2}Q_{\ell,N}(t)^{1/2},
\end{equation}
where
\begin{align*}
\varepsilon_{\ell,N}(t)
&:=\left|\sum_{i=1}^{k_\ell-1}\int_{t_i^\ell}^{t_{i+1}^\ell}\int_{\mathbb T^d}
A(s)\cdot\nabla\bigl(u(t_{i+1}^\ell)-u(t_i^\ell)-\Pi_N\Delta_i\mathbb M\bigr)\dd x\dd s\right|,\\
Q_{\ell,N}(t)&:=\sum_{i=1}^{k_\ell-1}\norm{(I-\Pi_N)\Delta_i\mathbb M}_H^2.
\end{align*}
For fixed $N$,~\eqref{eq:px-step-grad} gives
\[
\int_{Q_T}|A|\,|\nabla(\widehat u^\ell-\widetilde u^\ell)|\dd\mathbb P\dd s\dd x\longrightarrow0.
\]
The increments of $\Pi_N\mathbb M$ converge uniformly to zero in $V$, almost surely. Since $\int_0^t\|\Div A(s)\|_{V^*}\dd s<\infty$ almost surely, the contribution of these terms also vanishes. Hence $\varepsilon_{\ell,N}(t)\to0$ in probability.

By Lemma~\ref{lem:Hilbert-discrete-qv}, the full sum including $i=0$ converges in probability to
\[
Q_N(t):=\int_0^t\|(I-\Pi_N)B(u(s))\|_{\mathcal L_2(\mathcal U,H)}^2\dd s.
\]
The omitted first increment tends to zero almost surely by continuity of $\mathbb M$. Therefore
\begin{equation*}
Q_{\ell,N}(t)\longrightarrow Q_N(t)\quad\text{in probability}.
\end{equation*}

Given an arbitrary subsequence in $\ell$, choose a further diagonal subsequence along which these two convergences hold almost surely for every $N$. Setting $x_\ell=(\mathscr R_\ell(t))^{1/2}$, we obtain from~\eqref{eq:remainder-projection-expanded} that
\[
\limsup_{\ell\to\infty}x_\ell\le Q_N(t)^{1/2}\qquad\text{a.s. for every }N.
\]
Since $\Pi_N\to I$ strongly on $H$, dominated convergence for Hilbert--Schmidt operators yields $Q_N(t)\to0$ almost surely. Consequently, every subsequence contains a further subsequence along which $\mathscr R_\ell(t)\to0$ almost surely; this is equivalent to~\eqref{eq:internal-remainder-zero}.
\end{proof}

The following energy formula applies before identification of the nonlinear flux.
\begin{proposition}\label{prop:variable-Ito}
Let $u_0\in L^2(\Omega,\mathcal F_0,\mathbb P;H)$ and let $u$ be an adapted process with weakly continuous paths in $H$ such that
\begin{equation}\label{eq:Ito-sup-assumption}
\mathbb{E}\sup_{0\le t\le T}\norm{u(t)}_H^2<\infty.
\end{equation}
Assume that, in $V^*$,
\begin{equation}\label{eq:abstract-unidentified}
u(t)=u_0+\int_0^t\Div F(s)\dd s
+\int_0^t\mathcal C_\xi(u(s))\dd s
+\int_0^tB(u(s))\dd\mathbb W_s,
\end{equation}
where
\[
\nabla u\in L^{\widetilde p}(Q_T;\mathbb R^d),
\qquad
F\in L^{\widetilde p'}(Q_T;\mathbb R^d),
\qquad
B(u)\in L^2_{\mathcal P}(\Omega\times(0,T);\mathcal L_2(\mathcal U,H)),
\]
and $\mathcal C_\xi(u)$ is given by~\eqref{eq:Cxi-pairing}.  Then $u$ has a modification with strongly continuous paths in $H$ and, on one set of probability one, for every $t\in[0,T]$,
\begin{align}
\norm{u(t)}_H^2
={}&\norm{u_0}_H^2
-2\int_0^t\int_{\mathbb T^d}F\cdot\nabla u\dd x\dd s
+2\int_0^t\pair{\mathcal C_\xi(u)}{u}\dd s\notag\\
&+\int_0^t\norm{B(u)}_{\mathcal L_2(\mathcal U,H)}^2\dd s
+2\int_0^t\pair{u(s)}{B(u(s))\dd\mathbb W_s}_H.
\label{eq:variable-Ito-formula}
\end{align}
For the transport operators~\eqref{eq:Cxi-general}--\eqref{eq:B-general}, the last three terms cancel identically, and hence
\begin{equation}\label{eq:unidentified-energy}
\norm{u(t)}_H^2
+2\int_0^t\int_{\mathbb T^d}F(s,x)\cdot\nabla u(s,x)\dd x\dd s
=\norm{u_0}_H^2.
\end{equation}
\end{proposition}

\begin{proof}
By Lemma~\ref{lem:transport-divergence}, with
$
A:=F+\mathfrak K(u)\in L^{\widetilde p'}(Q_T;\mathbb R^d),
$
we may rewrite~\eqref{eq:abstract-unidentified} as~\eqref{eq:abstract-divergence-total}, where
\[
\mathbb M(t):=\int_0^tB(u(s))\dd\mathbb W_s.
\]
Choose a full-measure set $\mathcal R\subset(0,T)$ such that $u(t)\in V$ almost surely for every $t\in\mathcal R$. Lemma~\ref{lem:right-endpoint-increment} then gives~\eqref{eq:discrete-increment-identity} on every interval whose right endpoint belongs to $\mathcal R$.

\noindent\textbf{Step 1: Time discretization.}
Choose the nested partitions from Lemma~\ref{lem:px-time-approx} with all interior nodes in $\mathcal R$. In addition to the $V$-valued left step process $\widetilde u^\ell$ from that lemma, define on $(0,T]$ the $H$-valued predictable process
\begin{equation*}
\widetilde u_H^\ell(s)
:=u_0\mathbf 1_{(0,t_1^\ell]}(s)+\widetilde u^\ell(s),
\end{equation*}
and set $\widetilde u_H^\ell(0):=u_0$. Since the two processes differ only on the first cell,
\[
\mathbb{E}\int_0^T\norm{\widetilde u_H^\ell-\widetilde u^\ell}_H^2\dd s
=t_1^\ell\mathbb{E}\norm{u_0}_H^2\longrightarrow0,
\]
and therefore~\eqref{eq:px-step-H} implies
\begin{equation}\label{eq:left-step-H-conv}
\widetilde u_H^\ell\longrightarrow u
\quad\text{strongly in }L^2(\Omega\times(0,T);H).
\end{equation}
Set
\[
\mathcal N:=\bigcup_{\ell\ge1}\bigl(I_\ell\cap(0,T)\bigr).
\]
Then $\mathcal N$ is countable and dense in $(0,T)$.  Fix $t\in\mathcal N$.  For all sufficiently large $\ell$, write
\[
0=t_0^\ell<t_1^\ell<\cdots<t_{k_\ell}^\ell=t.
\]
Summing~\eqref{eq:discrete-increment-identity} over the intervals of the partition gives the exact identity
\begin{align}
\norm{u(t)}_H^2-\norm{u_0}_H^2
={}&-2\int_0^t\int_{\mathbb T^d}A(s)\cdot\nabla\widehat u^\ell(s)\dd x\dd s
+2\int_0^t\pair{\widetilde u_H^\ell(s)}{B(u(s))\dd\mathbb W_s}_H\notag\\
&+\sum_{i=0}^{k_\ell-1}
\norm{\mathbb M(t_{i+1}^\ell)-\mathbb M(t_i^\ell)}_H^2
-r_\ell^0-\mathscr R_\ell(t),
\label{eq:partition-energy-refined}
\end{align}
where
\begin{equation}\label{eq:first-remainder}
r_\ell^0:=\norm{u(t_1^\ell)-u_0-\mathbb M(t_1^\ell)}_H^2,
\end{equation}
and $\mathscr R_\ell(t)$ denotes the internal remainder defined in~\eqref{eq:internal-remainder}. The first cell is handled separately, because $u_0$ need not lie in $V$.

The drift term in~\eqref{eq:partition-energy-refined} converges in $L^1(\Omega)$ to
$-2\int_0^t\int A\cdot\nabla u$, by virtue of~\eqref{eq:px-step-grad}. For the stochastic term, set $\tau_M:=\inf\{r\in[0,T]:\norm{u(r)}_H>M\}\wedge T$, with $\inf\varnothing=\infty$. This is a stopping time: adapted weakly continuous paths are progressively measurable in the separable space $H$, and $\{v:\norm{v}_H>M\}$ is weakly open. By~\eqref{eq:left-step-H-conv}, a subsequence satisfies $\widetilde u_H^\ell\to u$ almost everywhere on $\Omega\times(0,T)$. For $0<s<\tau_M$, both $u(s)$ and the left endpoint value $\widetilde u_H^\ell(s)$ have norm at most $M$. If $\norm{u_0}_H>M$, then $\tau_M=0$ and the stopped integrals vanish. Values at $s=\tau_M$ do not affect the time integrals. Dominated convergence with the integrable majorant $4M^2\norm{B(u)}_{\mathcal L_2(\mathcal U,H)}^2$ therefore gives
\[
\mathbb{E}\int_0^{t\wedge\tau_M}
\norm{\widetilde u_H^\ell-u}_H^2
\norm{B(u)}_{\mathcal L_2(\mathcal U,H)}^2\dd s\longrightarrow0.
\]
Applying the Burkholder--Davis--Gundy inequality yields
\[
\mathbb{E}\sup_{r\le t}\left|
\int_0^{r\wedge\tau_M}\pair{\widetilde u_H^\ell-u}{B(u)\dd\mathbb W}_H
\right|^2\longrightarrow0.
\]
Finally, $\tau_M\uparrow T$ almost surely because $u$ has weakly continuous $H$-paths and satisfies~\eqref{eq:Ito-sup-assumption}; more explicitly,
\[
\mathbb P(\tau_M<T)
\le M^{-2}\mathbb{E}\sup_{0\le r\le T}\norm{u(r)}_H^2
\longrightarrow0.
\]
Letting $M\to\infty$ gives
\[
\sup_{r\le t}\left|
\int_0^r\pair{\widetilde u_H^\ell-u}{B(u)\dd\mathbb W}_H
\right|\longrightarrow0
\quad\text{in probability}.
\]
By Lemma~\ref{lem:Hilbert-discrete-qv}, the quadratic-variation sums converge in probability to
\[
\int_0^t\norm{B(u(s))}_{\mathcal L_2(\mathcal U,H)}^2\dd s.
\]
Moreover, by Lemma~\ref{lem:internal-remainder}, $\mathscr R_\ell(t)\to0$ in probability.

Enumerate $\mathcal N=\{s_1,s_2,\ldots\}$. Choose $\ell_m>\ell_{m-1}$ such that the probability that any of the four errors at $s_1,\ldots,s_m$ exceeds $2^{-m}$ is at most $2^{-m}$. By the Borel--Cantelli lemma, all errors at every node converge almost surely along this subsequence. Intersecting countably many such events with the increment identities and the continuity events, we obtain a single full-probability set $\Omega_0$; relabel the partitions as $I_\ell$. Since $r_\ell^0\ge0$,~\eqref{eq:partition-energy-refined} yields, on $\Omega_0$ and for every $t\in\mathcal N$,
\begin{align}
\norm{u(t)}_H^2
\le{}&\norm{u_0}_H^2
-2\int_0^t\int_{\mathbb T^d}A\cdot\nabla u\dd x\dd s
+\int_0^t\norm{B(u)}_{\mathcal L_2(\mathcal U,H)}^2\dd s\notag\\
&+2\int_0^t\pair{u(s)}{B(u(s))\dd\mathbb W_s}_H.
\label{eq:preliminary-energy-inequality}
\end{align}

Let $s_\ell:=t_1^\ell$. Then $s_\ell\downarrow0$ along a subsequence, and $s_\ell\in\mathcal N$. By weak $H$ continuity,
$u(s_\ell)\rightharpoonup u_0$ weakly in $H$ almost surely. Since $A\cdot\nabla u$ and $\|B(u)\|_{\mathcal L_2}^2$ are integrable, and since the scalar local martingale is continuous,~\eqref{eq:preliminary-energy-inequality} yields
\[
\limsup_{\ell\to\infty}\norm{u(s_\ell)}_H^2\le\norm{u_0}_H^2.
\]
Weak lower semicontinuity gives the reverse liminf inequality, and therefore
\begin{equation}\label{eq:strong-initial-nodes}
u(s_\ell)\longrightarrow u_0
\quad\text{strongly in }H\quad\text{almost surely}.
\end{equation}
Since $\mathbb M(s_\ell)\to0$ in $H$ by continuity, it follows from~\eqref{eq:first-remainder} and~\eqref{eq:strong-initial-nodes} that
\begin{equation}\label{eq:first-remainder-zero}
r_\ell^0\longrightarrow0
\qquad\text{almost surely}.
\end{equation}
Returning to the exact identity~\eqref{eq:partition-energy-refined} and using
\eqref{eq:internal-remainder-zero}--\eqref{eq:first-remainder-zero}, we obtain that for every $t\in\mathcal N$,
\begin{align}
\norm{u(t)}_H^2
={}&\norm{u_0}_H^2
-2\int_0^t\int_{\mathbb T^d}A\cdot\nabla u\dd x\dd s
+\int_0^t\norm{B(u)}_{\mathcal L_2(\mathcal U,H)}^2\dd s\notag\\
&+2\int_0^t\pair{u(s)}{B(u(s))\dd\mathbb W_s}_H.
\label{eq:px-Ito-total-flux}
\end{align}

By~\eqref{eq:px-step-grad} and the localized martingale convergence, take a further subsequence $I_m$ and refine $\Omega_0$ so that
\begin{equation}\label{eq:global-step-errors-as}
\int_0^T\int_{\mathbb T^d}
\left(
|\nabla\widetilde u^m-\nabla u|^{p(x)}
+|\nabla\widehat u^m-\nabla u|^{p(x)}
\right)\dd x\dd s\longrightarrow0,
\end{equation}
and
\[
\sup_{r\le T}\left|
\int_0^r\pair{u-\widetilde u_H^m}{B(u)\dd\mathbb W}_H
\right|\longrightarrow0
\quad\text{almost surely}.
\]
Since $\mathcal N$ is countable, we further reduce $\Omega_0$ so that, for every $a,b\in\mathcal N$ with $a<b$, the increment equation tested by $u(a)$ and the identity
\[
\pair{u(a)}{\mathbb M(b)-\mathbb M(a)}_H
=\int_a^b\pair{u(a)}{B(u(s))\dd\mathbb W_s}_H
\]
hold simultaneously. The latter identity follows by approximation with simple $\mathcal F_a$-measurable $H$-valued random variables and localization.

For $t\in(0,T)$ let $t_m$ be the largest node of $I_m$ not exceeding $t$; for $t=T$ use the last interior node. For large $m$, $t_m\in\mathcal R$, $t_m\uparrow t$, and $\widetilde u^m=u(t_m)$ on $(t_m,t_\ell]$ for $\ell>m$. Subtract~\eqref{eq:px-Ito-total-flux} at these nodes and test the increment equation by $u(t_m)$ to obtain
\begin{align*}
\norm{u(t_\ell)-u(t_m)}_H^2
={}&-2\int_{t_m}^{t_\ell}\int_{\mathbb T^d}
A(s)\cdot\nabla\bigl(u(s)-\widetilde u^m(s)\bigr)\dd x\dd s\\
&+2\int_{t_m}^{t_\ell}
\pair{u(s)-\widetilde u_H^m(s)}{B(u(s))\dd\mathbb W_s}_H
+\int_{t_m}^{t_\ell}\norm{B(u(s))}_{\mathcal L_2(\mathcal U,H)}^2\dd s.
\end{align*}
On $\Omega_0$, by Fubini's theorem, variable-exponent H\"older and \eqref{eq:global-step-errors-as}, we obtain
\begin{align*}
&~~\left|\int_{t_m}^{t_\ell}\int_{\mathbb T^d}A(s,\omega)\cdot\nabla(u-\widetilde u^m)(s,\omega)\dd x\dd s\right|\\
&\le2\norm{A(\omega)}_{L^{p'(\cdot)}((0,T)\times\mathbb T^d)}
\norm{\nabla u(\omega)-\nabla\widetilde u^m(\omega)}_{L^{p(\cdot)}((0,T)\times\mathbb T^d)}\longrightarrow0.
\end{align*}
Note that this bound is uniform in $t$ and $\ell>m$. The martingale term is bounded by twice its supremum, and absolute continuity gives
\[
\sup_{\substack{0\le a\le b\le T\\b-a\le|I_m|}}
\int_a^b\|B(u(s))\|_{\mathcal L_2}^2\dd s\longrightarrow0.
\]
Thus $u(t_m)$ is Cauchy in $H$. By weak continuity, its limit is identified with $u(t)$, so~\eqref{eq:px-Ito-total-flux} holds for every $t\in[0,T]$ on $\Omega_0$. Since the right-hand side is continuous, and since the norm is continuous and $u$ is weakly $H$-continuous, we conclude that $u\in C([0,T];H)$.

\noindent\textbf{Step 2: Transport cancellation.}
By~\eqref{eq:transport-flux-K} and~\eqref{eq:Cxi-pairing}, we obtain
\[
-2\int_{\mathbb T^d}\mathfrak K(u)\cdot\nabla u\dd x
=-\sum_j\norm{\xi_j\cdot\nabla u}_H^2
=2\pair{\mathcal C_\xi(u)}{u}.
\]
Consequently,~\eqref{eq:px-Ito-total-flux} with $A=F+\mathfrak K(u)$ coincides with~\eqref{eq:variable-Ito-formula}. Finally,
\[
\norm{B(u)}_{\mathcal L_2(\mathcal U,H)}^2
=\sum_j\norm{\xi_j\cdot\nabla u}_H^2,
\qquad
\pair{u}{B(u)\mathbf e_j}_H
=\frac12\int_{\mathbb T^d}\xi_j\cdot\nabla(u^2)\dd x=0.
\]
Therefore the transport correction, the quadratic variation, and the scalar martingale cancel, and~\eqref{eq:unidentified-energy} follows. This completes the proof.
\end{proof}

\subsection*{Funding}
J. Zhang is supported
by the China Postdoctoral Science Foundation under Grant Number 2026M794838.

\subsection*{Declaration of competing interest}
The authors declare that they have no known competing financial interests or personal relationships that could have appeared to influence the work reported in this paper.

\subsection*{Data availability}
No datasets were generated or analysed in this theoretical study.

\subsection*{Author Contributions}

All authors contributed equally to the conception, analysis, and preparation of this manuscript. All authors have read and approved the final manuscript.

\small
\phantomsection

\end{document}